\documentclass[graybox]{SNmult}

\usepackage{algorithm}
\usepackage{algpseudocode}
\usepackage{mathtools}
\usepackage{amsfonts}
\usepackage{amsmath}
\usepackage{amssymb}
\usepackage{dsfont}
\usepackage{setspace}
\usepackage{mathrsfs}
\usepackage{yfonts}
\usepackage[T1]{fontenc}
\usepackage{caption}
\usepackage{subcaption}
\usepackage{tikz}
\usepackage{pgfplots}
\usepackage{placeins}
\usepackage{type1cm}        
\usepackage{makeidx}         
\usepackage{graphicx}        
\usepackage{multicol}        
\usepackage[bottom]{footmisc}

\usepackage{newtxtext}       %
\usepackage[varvw]{newtxmath}       

\definecolor{darkgray}{RGB}{105,105,105}  
\definecolor{skyblue}{RGB}{86,180,233}
\definecolor{bluishgreen}{RGB}{0,158,115}
\definecolor{vermillion}{RGB}{213,94,0}
\definecolor{reddishpurple}{RGB}{204,121,167}
\definecolor{orange}{RGB}{230,159,0}
\definecolor{blue}{RGB}{0,114,178}

\pgfplotscreateplotcyclelist{myCycleList}{
  {darkgray,     mark=*,         solid, thick},
  {skyblue,      mark=square*,   solid, thick},
  {vermillion,   mark=triangle*, solid, thick},
  {bluishgreen,  mark=diamond*,  solid, thick},
  {reddishpurple,mark=otimes*,   solid, thick},
  {blue,         mark=triangle*, solid, thick},
  {orange,       mark=star,      solid, thick},
}

\pgfplotscreateplotcyclelist{myCycleList2}{
  {skyblue,      mark=square*,   solid, thick},
  {vermillion,   mark=triangle*, solid, thick},
  {bluishgreen,  mark=diamond*,  solid, thick},
  {reddishpurple,mark=otimes*,   solid, thick},
  {blue,         mark=triangle*, solid, thick},
  {orange,       mark=star,      solid, thick},
}

\usepackage[backend=biber,style=alphabetic]{biblatex}
\makeindex             

\begin{document}
\title*{Asymptotic and pre-asymptotic convergence of sparse grids for anisotropic kernel interpolation}
\titlerunning{Anisotropic kernel interpolation}
\author{Elliot J. Addy\orcidID{0009-0003-2667-8493} and\\ Aretha L. Teckentrup\orcidID{0000-0001-5089-0476}}
\institute{Elliot J. Addy \at  School of Mathematics and Maxwell Institute for Mathematical Sciences, University of
Edinburgh, King’s Buildings, Edinburgh EH9 3FD, UK \& Biomathematics and Statistics Scotland  \email{e.j.addy@sms.ed.ac.uk}
\and Aretha L. Teckentrup \at School of Mathematics and Maxwell Institute for Mathematical Sciences, University of
Edinburgh, King’s Buildings, Edinburgh EH9 3FD, UK \email{a.teckentrup@ed.ac.uk}}
%
%
\maketitle
\abstract*{Sparse grids are popular tools for high-dimensional function approximation. In this work, we study the use of sparse grids for interpolation using separable Mat\'ern kernels \(\Phi_{\boldsymbol{\nu},\boldsymbol{\lambda}}(\mathbf{x},\mathbf{x}')=\prod_{j=1}^d\phi_{\nu_j,\lambda_j}(x_j,x_j')\), with a particular focus on the anisotropic setting where the regularity $\nu_j$ and the lengthscale $\lambda_j$ vary with dimension $j$. We combine the construction of anisotropic sparse grids, which exploit anisotropic $\nu_j$ to improve convergence rates in smooth dimensions, with the construction of lengthscale-informed sparse grids, which diminish the error contribution of less varying dimensions using anisotropic $\lambda_j$. We provide theory and numerical experiments to showcase the benefits on asymptotic and pre-asymptotic error behaviour of sparse grid kernel interpolation. 
}

\abstract{Sparse grids are popular tools for high-dimensional function approximation. In this work, we study the use of sparse grids for interpolation using separable Mat\'ern kernels \(\Phi_{\boldsymbol{\nu},\boldsymbol{\lambda}}(\mathbf{x},\mathbf{x}')=\prod_{j=1}^d\phi_{\nu_j,\lambda_j}(x_j,x_j')\), with a particular focus on the anisotropic setting where the regularity $\nu_j$ and the lengthscale $\lambda_j$ vary with dimension $j$. We combine the construction of anisotropic sparse grids, which exploit anisotropic $\nu_j$ to improve convergence rates in smooth dimensions, with the construction of lengthscale-informed sparse grids, which diminish the error contribution of less varying dimensions using anisotropic $\lambda_j$. We provide theory and numerical experiments to showcase the benefits on asymptotic and pre-asymptotic error behaviour of sparse grid kernel interpolation. 
}

\section{Introduction}\label{sec: introduction}

Many tasks in modern applied mathematics, including the construction of surrogates for large-scale mathematical models in uncertainty quantification, machine learning, and data-based scientific computing, require efficient approximation of high-dimensional functions. To mitigate the {\em curse of dimensionality}, where the computational effort to achieve a given accuracy increases exponentially with the dimension $d$, structural assumptions have to be imposed on the function $f$ as well as on the methodology used to approximate it. The focus of this work is on the setting where $f$ exhibits an anisotropic structure, manifested by higher regularity and/or less variation in some of the dimensions, and the construction of sparse grid kernel interpolants in this setting.

We will focus on the frequently used separable Mat\'ern kernels \(\Phi_{\boldsymbol{\nu},\boldsymbol{\lambda}}(\mathbf{x},\mathbf{x}')=\prod_{j=1}^d\phi_{\nu_j,\lambda_j}(x_j,x_j')\) as basis functions. Through the regularity parameters $\nu_j$ and lengthscale parameters $\lambda_j$, this easily allows for defining anisotropic structures: higher values of $\nu_j$ represent smoother dependence, and higher values of $\lambda_j$ indicate less variation. Anisotropic sparse grids (ASGs), as studied in  \cite{Nobile2008} for polynomial basis functions, exploit anisotropic $\nu_j$ and hence $j$-dependent convergence rates to slow down the growth rate of number of points placed in the smoother dimensions compared to its isotropic counterpart. Lengthscale-informed sparse grids (LISGs) \cite{addy2025lengthscaleinformedsparsegridskernel,addy2026thesis} in contrast use anisotropic $\lambda_j$ to delay the onset of the growth of number of points placed in the dimensions with less variation compared to its isotropic counterpart. Thus, although both constructions alleviate dimension dependence by placing fewer points in the dimensions that are easier to approximate, the mechanics are fundamentally different.

The main contribution of this work is to combine the construction, and hence the benefits, of anisotropic and lengthscale-informed sparse grids into {\em doubly anisotropic sparse grids} (DASGs). Through Theorem~\ref{thm: doubly anisotropic error}, we show how this construction allows for improved convergence rates in both pre-asymptotic and asymptotic regimes. More precisely, since ASGs slow down the growth rate of the number of points placed in certain dimensions of the sparse grid, this construction is expected to provide faster (and under suitable assumptions dimension independent) asymptotic convergence rates in the number of points $N$ compared to isotropic sparse grids. 
LISGs on the other hand delay the onset of the growth of number of points placed in certain dimensions but asymptotically have the same growth rate as their isotropic counterpart, and hence improve only on the pre-asymptotic behaviour. DASGs inherit the benefits of both. For completeness, we mention here that the pre-asymptotic regime in practice often appears to cover the range of values of $N$ of interest (see e.g \cite{addy2026thesis,addy2025lengthscaleinformedsparsegridskernel,Zech2020Smolyak}). 

Finally, we note that a fast implementation of DASGs is possible adapting the algorithms in \cite{Plumlee2014,addy2025lengthscaleinformedsparsegridskernel}, as presented in section \ref{sec:dasg implementation}.
We further emphasise that DASGs also help mitigate the notoriously ill-conditioned Gram matrices that arise in kernel interpolation, especially when employing large lengthscales or regularities.


\section{Mat\'ern kernel interpolation}\label{sec:matern}

We first introduce the background theory underpinning Mat\'ern kernel interpolation on uniform sparse grid designs. {Without loss of generality}, we consider target functions \(f:\Gamma^d\subset\mathbb{R}^d\rightarrow\mathbb{R}\) restricted to the \(d\)-dimensional unit cube; \(\Gamma=(-1/2,1/2)\).

\begin{definition}[ \cite{Lord_Powell_Shardlow_2014}]\label{def: 1D Matern}
\begin{enumerate}
    \item[(a)]
    For \(\nu,\lambda,\sigma\in\mathbb{R}_{>0}\), the one-dimensional \emph{Mat\'ern kernel}, \(\phi_{\nu,\lambda}:\Gamma\times\Gamma\rightarrow\mathbb{R}_{>0}\), is
    \begin{align}\label{eq:1d_matern}
        \phi_{\nu,\lambda}({x},{x}')\coloneqq\sigma^2\frac{2^{1-\nu}}{{\Gamma(\nu)}}\left(\sqrt{2\nu}\frac{|\,{x}-{x}'|}{\lambda}\right)^{\nu}K_{\nu}\left(\sqrt{2\nu}\frac{|\,{x}-{x}'|}{\lambda}\right),
    \end{align}
    for all \({x},{x}'\in\Gamma\), where \(K_\nu\) is the modified Bessel function of the second kind.\footnote{{In contrast to the notation used in \cite{addy2025lengthscaleinformedsparsegridskernel}, here we define these kernels over the unit cube.}} 
    \item[(b)] For \(d\in\mathbb{N}\) and \(\boldsymbol{\nu}\coloneqq(\nu_1,\dots,\nu_d)\), \(\boldsymbol{\lambda}\coloneqq(\lambda_1,\dots,\lambda_d)\) and \(\boldsymbol{\sigma}\coloneqq(\sigma_1,\dots,\sigma_d)\), we define the \emph{separable Mat\'ern kernel}, \(\Phi_{\boldsymbol{\nu},\boldsymbol{\lambda}}:\Gamma^d\times\Gamma^d\rightarrow\mathbb{R}_{>0}\), as the product of the \(d\)-many one-dimensional kernels, \(\Phi_{\boldsymbol{\nu},\boldsymbol{\lambda}}(\mathbf{x},\mathbf{x}')=\prod_{j=1}^d\phi_{\nu_j,\lambda_j}(x_j,x_j')\), for all \(\mathbf{x},\mathbf{x}'\in\Gamma^d\).
\end{enumerate}
\end{definition}
The parameters \(\nu\), \(\lambda\) and \(\sigma\) represent the regularity, lengthscale and scaling of the kernel, respectively.\footnote{We omit \(\sigma\) from the notation as the corresponding interpolants are independent of it.} Simplified expressions exist for half-integer valued \(\nu\), including the exponential kernel \(\phi_{1/2,\lambda}(x,x')=\sigma^2e^{-|\,x-x'|/\lambda}\), and, in the limit \(\nu\rightarrow\infty\), the squared exponential (or Gaussian) kernel \(\phi_{\infty,\lambda}(x,x')=\sigma^2e^{-|\,x-x'|^2/2\lambda^2}\). Mat\'ern kernels are particularly valuable tools for function approximation as their \emph{native spaces} -- a.~k.~a. \emph{reproducing kernel Hilbert spaces} -- coincide with Sobolev spaces , see e.g. \cite{wendland_2004}.

\begin{proposition}[Corollary 10.13, \cite{wendland_2004} and Theorem 1, Section 8, \cite{Aronszajn1950}]\label{prop: matern native spaces}
\begin{enumerate}
    \item[(a)]
    The native space associated with the Mat\'ern kernel \(\phi_{{\nu},{\lambda}}\), denoted \(\mathcal{N}_{\phi_{\nu,\lambda}}(\Gamma)\), is isomorphic to the standard Sobolev space \(H^{\nu+1/2}(\Gamma)\) ,
    \item[(b)]
    The native space associated with the separable Mat\'ern kernel \(\Phi_{\boldsymbol{\nu},\boldsymbol{\lambda}}\), denoted \(\mathcal{N}_{\Phi_{\boldsymbol{\nu},\boldsymbol{\lambda}}}(\Gamma^d)\), is isomorphic to the Sobolev space of dominating mixed smoothness, \(H_{\textnormal{mix}}^{\boldsymbol{\nu}+\mathbf{1}/2}(\Gamma^d)\coloneqq\widehat{\bigotimes}_{j=1}^d H^{\nu_j+1/2}(\Gamma)\).
\end{enumerate}
\end{proposition}

To define our kernel interpolants in one dimension, we restrict the interpolation points to uniformly-spaced designs as given in Definition \ref{def:chi_l^p} below and illustrated in Figure~\ref{fig: point sets}. Note that the penalty $p$ can be interpreted as a delayed onset of the growth of points in the point set \(\mathcal{X}_\ell = \mathcal{X}_\ell^0\), see \cite{addy2026thesis} for further details.
\begin{definition}[\cite{addy2025lengthscaleinformedsparsegridskernel}]\label{def:chi_l^p}
    Let \(\ell\in\mathbb{Z}\) and \(p\in\mathbb{N}_0\) be given. Denote by \(\mathcal{X}_\ell\) the set of \(2^{\ell+1}-1\) uniformly spaced points in \(\Gamma\); \(\mathcal{X}_\ell \coloneqq\left\{n/2^{\ell+1}\in\Gamma\,:\,n\in\mathbb{Z}\right\}\). The \(p\)-\emph{penalised point-set}, \(\mathcal{X}_{\ell}^p\subset\Gamma\), is then given by 
    \begin{align}
        \mathcal{X}_{\ell}^p\coloneqq\begin{cases}
            \mathcal{X}_{\ell-p}\quad&\textrm{ for }l\geq p+1,\\
            \mathcal{X}_{0}=\{0\}&\textrm{ for }0\leq \ell\leq p, \textrm{and}\\
            \emptyset&\textrm{ for }\ell<0.\end{cases}\nonumber
    \end{align}
\end{definition}
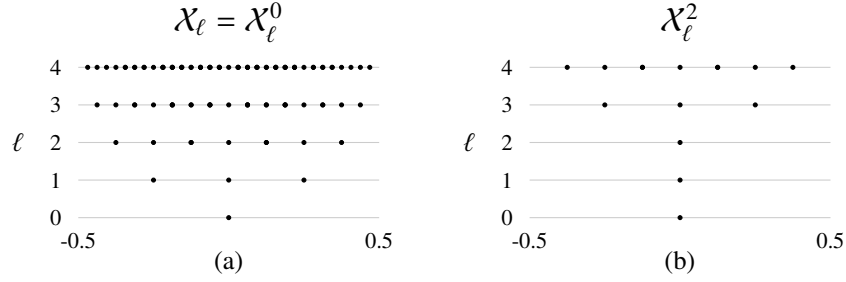
\begin{figure}[h]
    \centering
    \begin{tikzpicture}
        \node[black] at (-2-3,-0.3) {\small-0.5};
        \node[black] at (2-3,-0.3) {\small0.5};
        \node[black] at (-2.8-3,1) {\(\ell\)};
        \node[black] at (-2.3-3,-0) {\footnotesize0};
        \node[black] at (-2.3-3,0.5) {\footnotesize1};
        \node[black] at (-2.3-3,1) {\footnotesize2};
        \node[black] at (-2.3-3,1.5) {\footnotesize3};
        \node[black] at (-2.3-3,2) {\footnotesize4};
        \node[black] at (0-3,0-0.6) {(a)};
        \node[black] at (0-3,0+2.6) {\large\(\mathcal{X}_\ell=\mathcal{X}_\ell^0\)};
        \draw[gray!40] (-2-3,0) -- (2-3,0);
        \draw[gray!40] (-2-3,0.5) -- (2-3,0.5);
        \draw[gray!40] (-2-3,1) -- (2-3,1);
        \draw[gray!40] (-2-3,1.5) -- (2-3,1.5);
        \draw[gray!40] (-2-3,2) -- (2-3,2);
        
        \filldraw[black] (0-3,0) circle (3/4pt);
        \filldraw[black] (0-3,0+0.5) circle (3/4pt);
        \filldraw[black] (0-1-3,0+0.5)  circle (3/4pt);
        \filldraw[black] (0+1-3,0+0.5) circle (3/4pt);
        \filldraw[black] (0-3,0+1) circle (3/4pt);
        \filldraw[black] (0-1-3,0+1) circle (3/4pt);
        \filldraw[black] (0+1-3,0+1) circle (3/4pt);
        \filldraw[black] (0+0.5-3,0+1) circle (3/4pt);
        \filldraw[black] (0-1+0.5-3,0+1) circle (3/4pt);
        \filldraw[black] (0+1+0.5-3,0+1) circle (3/4pt);
        \filldraw[black] (0-0.5-3,0+1) circle (3/4pt);
        \filldraw[black] (0-1-0.5-3,0+1) circle (3/4pt);
        \filldraw[black] (0+1-0.5-3,0+1) circle (3/4pt);
        \filldraw[black] (0-3,0+1.5) circle (3/4pt);
        \filldraw[black] (0-1-3,0+1.5) circle (3/4pt);
        \filldraw[black] (0+1-3,0+1.5) circle (3/4pt);
        \filldraw[black] (0+0.5-3,0+1.5) circle (3/4pt);
        \filldraw[black] (0-1+0.5-3,0+1.5) circle (3/4pt);
        \filldraw[black] (0+1+0.5-3,0+1.5) circle (3/4pt);
        \filldraw[black] (0-0.5-3,0+1.5) circle (3/4pt);
        \filldraw[black] (0-1-0.5-3,0+1.5) circle (3/4pt);
        \filldraw[black] (0+1-0.5-3,0+1.5) circle (3/4pt);
        \filldraw[black] (0-0.25-3,0+1.5) circle (3/4pt);
        \filldraw[black] (0-1-0.25-3,0+1.5) circle (3/4pt);
        \filldraw[black] (0+1-0.25-3,0+1.5) circle (3/4pt);
        \filldraw[black] (0+0.5-0.25-3,0+1.5) circle (3/4pt);
        \filldraw[black] (0-1+0.5-0.25-3,0+1.5) circle (3/4pt);
        \filldraw[black] (0+1+0.5-0.25-3,0+1.5) circle (3/4pt);
        \filldraw[black] (0-0.5-0.25-3,0+1.5) circle (3/4pt);
        \filldraw[black] (0-1-0.5-0.25-3,0+1.5) circle (3/4pt);
        \filldraw[black] (0+1-0.5-0.25-3,0+1.5) circle (3/4pt);
        \filldraw[black] (0+0.25-3,0+1.5) circle (3/4pt);
        \filldraw[black] (0-1+0.25-3,0+1.5) circle (3/4pt);
        \filldraw[black] (0+1+0.25-3,0+1.5) circle (3/4pt);
        \filldraw[black] (0+0.5+0.25-3,0+1.5) circle (3/4pt);
        \filldraw[black] (0-1+0.5+0.25-3,0+1.5) circle (3/4pt);
        \filldraw[black] (0+1+0.5+0.25-3,0+1.5) circle (3/4pt);
        \filldraw[black] (0-0.5+0.25-3,0+1.5) circle (3/4pt);
        \filldraw[black] (0-1-0.5+0.25-3,0+1.5) circle (3/4pt);
        \filldraw[black] (0+1-0.5+0.25-3,0+1.5) circle (3/4pt);
        \filldraw[black] (0-3,0+2) circle (3/4pt);
        \filldraw[black] (0-1-3,0+2) circle (3/4pt);
        \filldraw[black] (0+1-3,0+2) circle (3/4pt);
        \filldraw[black] (0+0.5-3,0+2) circle (3/4pt);
        \filldraw[black] (0-1+0.5-3,0+2) circle (3/4pt);
        \filldraw[black] (0+1+0.5-3,0+2) circle (3/4pt);
        \filldraw[black] (0-0.5-3,0+2) circle (3/4pt);
        \filldraw[black] (0-1-0.5-3,0+2) circle (3/4pt);
        \filldraw[black] (0+1-0.5-3,0+2) circle (3/4pt);
        \filldraw[black] (0-0.25-3,0+2) circle (3/4pt);
        \filldraw[black] (0-1-0.25-3,0+2) circle (3/4pt);
        \filldraw[black] (0+1-0.25-3,0+2) circle (3/4pt);
        \filldraw[black] (0+0.5-0.25-3,0+2) circle (3/4pt);
        \filldraw[black] (0-1+0.5-0.25-3,0+2) circle (3/4pt);
        \filldraw[black] (0+1+0.5-0.25-3,0+2) circle (3/4pt);
        \filldraw[black] (0-0.5-0.25-3,0+2) circle (3/4pt);
        \filldraw[black] (0-1-0.5-0.25-3,0+2) circle (3/4pt);
        \filldraw[black] (0+1-0.5-0.25-3,0+2) circle (3/4pt);
        \filldraw[black] (0+0.25-3,2) circle (3/4pt);
        \filldraw[black] (0-1+0.25-3,2) circle (3/4pt);
        \filldraw[black] (0+1+0.25-3,2) circle (3/4pt);
        \filldraw[black] (0+0.5+0.25-3,2) circle (3/4pt);
        \filldraw[black] (0-1+0.5+0.25-3,0+2) circle (3/4pt);
        \filldraw[black] (0+1+0.5+0.25-3,0+2) circle (3/4pt);
        \filldraw[black] (0-0.5+0.25-3,0+2) circle (3/4pt);
        \filldraw[black] (0-1-0.5+0.25-3,0+2) circle (3/4pt);
        \filldraw[black] (0+1-0.5+0.25-3,0+2) circle (3/4pt);
        \filldraw[black] (0-0.25-0.125-3,0+2) circle (3/4pt);
        \filldraw[black] (0-1-0.25-0.125-3,0+2) circle (3/4pt);
        \filldraw[black] (0+1-0.25-0.125-3,0+2) circle (3/4pt);
        \filldraw[black] (0+0.5-0.25-0.125-3,0+2) circle (3/4pt);
        \filldraw[black] (0-1+0.5-0.25-0.125-3,0+2) circle (3/4pt);
        \filldraw[black] (0+1+0.5-0.25-0.125-3,0+2) circle (3/4pt);
        \filldraw[black] (0-0.5-0.25-0.125-3,0+2) circle (3/4pt);
        \filldraw[black] (0-1-0.5-0.25-0.125-3,0+2) circle (3/4pt);
        \filldraw[black] (0+1-0.5-0.25-0.125-3,0+2) circle (3/4pt);
        \filldraw[black] (0+0.25-0.125-3,0+2) circle (3/4pt);
        \filldraw[black] (0-1+0.25-0.125-3,0+2) circle (3/4pt);
        \filldraw[black] (0+1+0.25-0.125-3,0+2) circle (3/4pt);
        \filldraw[black] (0+0.5+0.25-0.125-3,0+2) circle (3/4pt);
        \filldraw[black] (0-1+0.5+0.25-0.125-3,0+2) circle (3/4pt);
        \filldraw[black] (0+1+0.5+0.25-0.125-3,0+2) circle (3/4pt);
        \filldraw[black] (0-0.5+0.25-0.125-3,0+2) circle (3/4pt);
        \filldraw[black] (0-1-0.5+0.25-0.125-3,0+2) circle (3/4pt);
        \filldraw[black] (0+1-0.5+0.25-0.125-3,0+2) circle (3/4pt);
        \filldraw[black] (0-0.25+0.125-3,0+2) circle (3/4pt);
        \filldraw[black] (0-1-0.25+0.125-3,0+2) circle (3/4pt);
        \filldraw[black] (0+1-0.25+0.125-3,0+2) circle (3/4pt);
        \filldraw[black] (0+0.5-0.25+0.125-3,0+2) circle (3/4pt);
        \filldraw[black] (0-1+0.5-0.25+0.125-3,0+2) circle (3/4pt);
        \filldraw[black] (0+1+0.5-0.25+0.125-3,0+2) circle (3/4pt);
        \filldraw[black] (0-0.5-0.25+0.125-3,0+2) circle (3/4pt);
        \filldraw[black] (0-1-0.5-0.25+0.125-3,0+2) circle (3/4pt);
        \filldraw[black] (0+1-0.5-0.25+0.125-3,0+2) circle (3/4pt);
        \filldraw[black] (0+0.25+0.125-3,0+2) circle (3/4pt);
        \filldraw[black] (0-1+0.25+0.125-3,0+2) circle (3/4pt);
        \filldraw[black] (0+1+0.25+0.125-3,0+2) circle (3/4pt);
        \filldraw[black] (0+0.5+0.25+0.125-3,0+2) circle (3/4pt);
        \filldraw[black] (0-1+0.5+0.25+0.125-3,0+2) circle (3/4pt);
        \filldraw[black] (0+1+0.5+0.25+0.125-3,0+2) circle (3/4pt);
        \filldraw[black] (0-0.5+0.25+0.125-3,0+2) circle (3/4pt);
        \filldraw[black] (0-1-0.5+0.25+0.125-3,0+2) circle (3/4pt);
        \filldraw[black] (0+1-0.5+0.25+0.125-3,0+2) circle (3/4pt);
        
        \node[black] at (-2+3,-0.3) {\small-0.5};
        \node[black] at (2+3,-0.3) {\small0.5};
        \node[black] at (-2.8+3,1) {\(\ell\)};
        \node[black] at (-2.3+3,-0) {\footnotesize0};
        \node[black] at (-2.3+3,0.5) {\footnotesize1};
        \node[black] at (-2.3+3,1) {\footnotesize2};
        \node[black] at (-2.3+3,1.5) {\footnotesize3};
        \node[black] at (-2.3+3,2) {\footnotesize4};
        \node[black] at (0+3,0-0.6) {(b)};
        \node[black] at (0+3,0+2.6) {\large\(\mathcal{X}_\ell^2\)};
        \draw[gray!40] (-2+3,0) -- (2+3,0);
        \draw[gray!40] (-2+3,0.5) -- (2+3,0.5);
        \draw[gray!40] (-2+3,1) -- (2+3,1);
        \draw[gray!40] (-2+3,1.5) -- (2+3,1.5);
        \draw[gray!40] (-2+3,2) -- (2+3,2);
        
        \filldraw[black] (0+3,0) circle (3/4pt);
        \filldraw[black] (0+3,0.5) circle (3/4pt);
        \filldraw[black] (0+3,1) circle (3/4pt);
        \filldraw[black] (0+3,0+1.5) circle (3/4pt);

        \filldraw[black] (0-1+3,0+1.5) circle (3/4pt);
        \filldraw[black] (0+1+3,0+1.5) circle (3/4pt);
        \filldraw[black] (0+3,0+2) circle (3/4pt);
        \filldraw[black] (0-1+3,0+2) circle (3/4pt);
        \filldraw[black] (0+1+3,0+2) circle (3/4pt);
        \filldraw[black] (0+0.5+3,0+2) circle (3/4pt);
        \filldraw[black] (0-1+0.5+3,0+2) circle (3/4pt);
        \filldraw[black] (0+1+0.5+3,0+2) circle (3/4pt);
        \filldraw[black] (0-0.5+3,0+2) circle (3/4pt);
        \filldraw[black] (0-1-0.5+3,0+2) circle (3/4pt);
        \filldraw[black] (0+1-0.5+3,0+2) circle (3/4pt);

    \end{tikzpicture}
    \caption{The nested point-sets \(\mathcal{X}_\ell\), (a), and \(\mathcal{X}_\ell^p\) with penalty \(p=2\), (b), for different levels, \(l\in\mathbb{N}_0\) (This is Figure 2, \cite{addy2025lengthscaleinformedsparsegridskernel}). 
    }
    \label{fig: point sets}
\end{figure}
\begin{definition}\label{def:restricted kernel interpolant}
    Let \(\mathcal{X}_\ell^p=\{{x}_1,\dots,{x}_N\}\subset\Gamma\) and define a linear sampling operator by \(T_{\mathcal{X}_\ell}(f)=f(\mathcal{X}_\ell)\coloneqq[f({x}_1),\dots,f({x}_N)]\). The one-dimensional \emph{Mat\'ern kernel interpolant} of a function \(f:\Gamma\rightarrow\mathbb{R}\), with respect to the point-set \(\mathcal{X}_\ell^p\), is defined by
   \begin{align}
       s_{\mathcal{X}_\ell^p,\phi_{\nu,\lambda}}(f)\coloneqq\underset{\substack{g\in\mathcal{N}_{\phi_{\nu,\lambda}}(\Gamma)\\T_{\mathcal{X}_\ell^p}(g)=T_{\mathcal{X}_\ell^p}(f)}}{\arg\min}\|g\|_{\mathcal{N}_{\phi_{\nu,\lambda}}(\Gamma)}.\nonumber
   \end{align}
\end{definition}
The value of the Mat\'ern kernel interpolant at \(x\in\Gamma\) is given by \(s_{\mathcal{X}_\ell^p,\phi_{\nu,\lambda}}(f)(x)=\mathbf{w}^\top\phi_{\nu,\lambda}(x, \mathcal{X}_{\ell}^p)\), where the coefficient vector \(\mathbf{w}\in\mathbb{R}^N\) is the solution of the linear system \(\phi_{\nu,\lambda}(\mathcal{X}_{\ell}^p,\mathcal{X}_{\ell}^p)\mathbf{w}=f(\mathcal{X}_{\ell}^p)\). Here, by fixing an ordering \(\mathcal{X}_\ell^p=\{x_1,\dots,x_N\}\), we define \(\phi_{\nu,\lambda}(x, \mathcal{X}_{\ell}^p)_i=\phi_{\nu,\lambda}(x, x_i)\) and \(\phi_{\nu,\lambda}(\mathcal{X}_{\ell}^p,\mathcal{X}_{\ell}^p)_{i,j}=\phi_{\nu,\lambda}(x_i,x_j)\). {We note that such systems can become notoriously ill-conditioned as \(N\), \(\nu\) and \(\lambda\) grow (see e.g. \cite{wendland_2004}).}

{Although presented for interpolation in one dimension, Definition \ref{def:restricted kernel interpolant} extends straightforwardly to higher dimensions. Without structural assumptions on \(f\), however, a manifestation of the \emph{curse of dimensionality} results in convergence rates deteriorating with $\frac{1}{d}$ and rapidly renders the methods computationally prohibitive in high dimensions. 


\section{Sparse grid constructions}\label{sec:sparse}
If \(f\) belongs to a Sobolev space of dominating mixed smoothness, one can construct interpolants whose worst-case errors exhibit a significantly weaker dependence on the  dimension \(d\). In particular, in this work we employ sparse grid approximations.
\begin{definition}\label{def: sparse grid operator}
Let \(\mathbf{p}\in\mathbb{N}_0^d\) and \(\boldsymbol{\nu},\boldsymbol{\lambda},\boldsymbol{\sigma}\in\mathbb{R}_0^d\), and let \(\mathcal{I}\subset\mathbb{N}_0^d\) be a finite collection of \emph{multi-indices} \(\boldsymbol{\ell}\in\mathcal{I}\). The corresponding \emph{Mat\'ern kernel sparse grid approximation operator}, based on the uniformly spaced point-sets \(\mathcal{X}_\ell^p\), is given by
\begin{align}
S_{\mathcal{I},\mathbf{p}}^{(\boldsymbol{\nu},\boldsymbol{\lambda})}\coloneqq\sum_{\boldsymbol{\ell}\in\mathcal{I}}\bigotimes_{j=1}^d\left(s_{\mathcal{X}_{\ell_j}^{p_j},\phi_{\nu_j,\lambda_j}}-s_{\mathcal{X}_{\ell_j-1}^{p_j},\phi_{\nu_j,\lambda_j}}\right).\nonumber
\end{align}
\end{definition}
Since the one-dimensional point-sets are nested, i.e. \(\mathcal{X}_\ell^p\subset\mathcal{X}_{\ell-1}^p\), by additionally enforcing that the multi-index set \(\mathcal{I}\) is \emph{downward closed}; that if \(\boldsymbol{\ell}=(l_1,\dots,l_d)\in\mathcal{I}\), and \(\ell_j\geq1\) for some \(1\leq j\leq d\), then \((\ell_1,\dots,\ell_j-1,\dots,\ell_d)\in\mathcal{I}\), we have \cite{Nobile2018}:
\begin{itemize}
    \item The sparse grid approximation is a kernel interpolant with respect to the kernel \(\Phi_{\boldsymbol{\nu},\boldsymbol{\lambda}}\), defined on a sparse grid design given by
    \(
    \mathcal{X}_{\mathcal{I}}^\otimes\coloneqq\sum_{\boldsymbol{\ell}\in\mathcal{I}}\mathcal{X}_{\ell_1}^{p_1}\otimes\cdots\otimes\mathcal{X}_{\ell_d}^{p_d}.
    \)
    \item The approximation error \(I-S_{\mathcal{I},\mathbf{p}}^{(\boldsymbol{\nu},\boldsymbol{\lambda})}\) in the mixed Sobolev norm may be characterised by a sum over the components corresponding to the excluded multi-indices \(\mathbb{N}_0^d\setminus\mathcal{I}\).
\end{itemize}
For the rest of this work, we identify different sparse grid interpolation methodologies with different choices of the multi-index set \(\mathcal{I}\subset\mathbb{N}_0^d\), the penalty vector \(\mathbf{p}\in\mathbb{N}_0^d\), and the kernel hyperparameters \(\boldsymbol{\nu},\boldsymbol{\lambda}\in\mathbb{R}_{>0}^d\).
As the simplest construction considered, we define \emph{isotropic sparse grids} (ISGs) given by \(\mathbf{p}=\mathbf{0}\), \(\boldsymbol{\lambda}=\mathbf{1}\), and, for some \(L \in \mathbb N_0\), 
\begin{align}
\mathcal{I}_{L}\coloneqq\{\boldsymbol{\ell}\in\mathbb{N}_0^d\,:\,|\boldsymbol{\ell}|_1\leq L\}.\label{eq: isotropic index set}
\end{align}
Thus, the level parameter \(L\) determines the number of interpolation points, and all directions \(1\leq j\leq d\) are associated with the same lengthscale and penalty value. An example ISG \(\mathcal{X}^\otimes_{\mathcal{I}_{L}}\) is depicted in two dimensions in Figure~\ref{fig: ISG and LISG}. 
For ISGs, we have the following asymptotic bound on the approximation error.

\begin{theorem}[Theorem 1, \cite{Nobile2018}, as written in Theorem 3.11, \cite{Teckentrup2020}]\label{thm: nobile result}
    Let \(0\leq\alpha_j\leq\nu_j\) and  \(\nu_j\geq1/2\) for all \(1\leq j\leq d\). For large enough \(L\in\mathbb{N}_0\), there exists a constant \(C_{\textnormal{T}\ref{thm: nobile result}}\), independent of \(L\), such that
    \begin{align}
        \left\lVert\, I-S_{\mathcal{I}_L,\mathbf{0}}^{(\boldsymbol{\nu},\mathbf{1})}\,\right\rVert_{H^{\boldsymbol{\nu}+1/2}_{\textnormal{mix}}(\Gamma^d)\rightarrow H^{\boldsymbol{\alpha}+1/2}_{\textnormal{mix}}(\Gamma^d)}&\leq \epsilon^{(d)}_{\boldsymbol{\nu},\boldsymbol{\alpha}}(L)\coloneqq C_{\textnormal{L}\ref{lem: supersets double}}^{(\boldsymbol{\nu},\boldsymbol{\alpha})}\sum_{\boldsymbol{\ell}\in\mathbb{N}_0^k\setminus\mathcal{I}_{L}}2^{-c|\boldsymbol{\ell}|_1}\nonumber,\\
        &\leq C_{\textnormal{T}\ref{thm: nobile result}}N_{L,d}^{-b_{\min}}(\log N_{L,d})^{(1+b_{\min})(d-1)},
    \end{align}
    where \(N_{L, d}\coloneqq|\mathcal{X}^{\otimes}_{\mathcal{I}_L}|\), \(b_{\min}=\min_{1\leq j\leq d}\nu_j-\alpha_j\),  and the constant \(C_{\textnormal{L}\ref{lem: supersets double}}^{(\boldsymbol{\nu},\boldsymbol{\alpha})}\) is given explicitly in Lemma~\ref{lem: supersets double}.
\end{theorem}

The convergence rate of the approximation error now depends on the dimension \(d\) only through a logarithmic factor, and is instead primarily governed by the smallest relative smoothness \(\nu_j-\alpha_j\) across all directions \(j\). Practically, this significantly expands the applicability of the methodology to higher-dimensional problems, {with our numerical experiments indicating suitability up to \(d\approx10\).

{Imposing further structural assumptions on the target function  \(f\) allows for alternative sparse grid constructions whose errors are even less sensitive to the dimension. In particular, a widely exploited structure in the high-dimensional approximation literature is \emph{anisotropy}, whereby \(f\) exhibits greater sensitivity to some parameters than others. As discussed in Section~\ref{sec: introduction}, there exist two main types of anisotropy studied:}
\begin{enumerate}
\item[(a)] Anisotropic regularity, where the number of derivatives of the function, and therefore the asymptotic convergence rate of the error, varies between coordinate directions.
\item[(b)] Weighted spaces, where variation along certain directions is penalised more than others, often through a weighted norm.
\end{enumerate}

In the context of interpolation with Mat\'ern kernels, these two types of anisotropy may be encoded in the hyperparameters \(\boldsymbol{\nu}\) and \(\boldsymbol{\lambda}\), respectively. Along with the use of anisotropic kernels, we may design more efficient sparse grid constructions that reflect this structure, allocating resolution preferentially to directions associated with either (a) lower smoothness or (b) greater variation. In principle, for a perfectly adapted method, it is the degree of the anisotropy present---not the  dimension \(d\)---that determines the computational complexity of the problem. {Additionally, by distributing interpolation points in this manner, such designs mitigate the increased ill-conditioning that would otherwise be induced by using large values of \(\nu\) or \(\lambda\).}

\subsection{Anisotropic sparse grids}\label{subsec: anisotropic sparse grids}

{Anisotropic regularity has been studied extensively for sparse grids in different contexts \cite{bungartz2004sparse,Nobile2008,Rieger2019}}. In Theorem~\ref{thm: nobile result}, we observed that the convergence rate of the error for ISGs is determined by 
the \textit{least smooth} direction. ASGs aim to improve on this by redistributing points toward the less smooth directions, effectively balancing the asymptotic convergence rates across the dimensions. This rebalancing, known as the \textit{anisotropic Smolyak} construction, is achieved by replacing the 1-norm in the construction of the isotropic multi-index set \(\mathcal{I}_{L}\) in \eqref{eq: isotropic index set} with a weighted sum. We define the anisotropic index set associated to the weight vector \(\boldsymbol{\omega}\in\mathbb{R}_{>0}^d\) by
\begin{align}
\mathcal{A}_{L,\boldsymbol{\omega}}\coloneqq\{\boldsymbol{\ell}\in\mathbb{N}_0^d\,:\,\boldsymbol{\omega}\cdot\boldsymbol{\ell}\leq L\}.\label{eq: anisotropic index set}
\end{align}
In Figure~\ref{fig: ASG and DASG}, an example ASG is illustrated in two dimensions with \(\boldsymbol{\omega}=(2,1)\). 

{In \cite{Nobile2008}, ASG approximations employing polynomial basis functions were shown to yield \(L^{\infty}\)-error convergence rates in \(N\) that are independent of the  dimension, \(d\), under suitable assumptions of increasing regularity with respect to the coordinate index \(j\)}. To the best of the authors' knowledge, an analogous asymptotic analysis in \(N\) has not yet been established for kernel interpolation. However, it is expected that the convergence rate will be similarly improved over that of Theorem~\ref{thm: nobile result}. A proof of such a result would require deriving suitable bounds for the operator norm 
\begin{align}
    \left\lVert\, I-S_{\mathcal{A}_{L,\boldsymbol{\omega}},\mathbf{0}}^{(\boldsymbol{\nu},\boldsymbol{\lambda})}\,\right\rVert_{H_{\textrm{mix}}^{\boldsymbol{\nu}}(\Gamma^d)\rightarrow H_{\textrm{mix}}^{\boldsymbol{\alpha}}(\Gamma^d)}&\leq \varepsilon^{(d)}_{\boldsymbol{\nu},\boldsymbol{\alpha},\boldsymbol{\omega}}(L)\coloneqq C_{\textnormal{L}\ref{lem: supersets double}}^{(\boldsymbol{\nu},\boldsymbol{\alpha})}\sum_{\boldsymbol{l}\in\mathbb{N}_0^d\setminus\mathcal{A}_{L,\boldsymbol{\omega}}}2^{-(\boldsymbol{\nu}-\boldsymbol{\alpha})\cdot\boldsymbol{l}},\label{eq: anisotropic bound}
\end{align}
where we define \(\varepsilon^{(d)}_{\boldsymbol{\nu},\boldsymbol{\alpha},\boldsymbol{\omega}}(L)\coloneqq0\) for \(L\leq0\). This formulation was derived in the proof of Theorem 4, \cite{addy2026thesis}. 
Empirical tests in \cite{addy2026thesis} suggest the weighting \(\omega_j=\nu_j-\alpha_j+1\), which we adopt in the numerical experiments in Section~\ref{sec: double experiments}.

\subsection{Lengthscale-informed sparse grids}\label{subsec: LISGs}

\begin{figure}
    \centering
    \begin{tikzpicture}[scale=0.8]


    \draw[gray, thick] (-5,5.5+-2) -- (-5,5.5+2);
    \draw[gray, thick] (-1,5.5+-2) -- (-1,5.5+2);
    \draw[gray, thick] (-5,5.5+2) -- (-1,5.5+2);
    \draw[gray, thick] (-5,5.5+-2) -- (-1,5.5+-2);
    \filldraw[gray, fill opacity = 0.2] (-5,5.5+2) -- (-1,5.5+2) -- (-1,5.5+-2) -- (-5,5.5+-2) -- (-5,5.5+2);
    \filldraw[white] (-4,5.5+0.5) -- (-2,5.5+0.5) -- (-2,5.5+-0.5) -- (-4,5.5+-0.5) -- (-4,5.5+0.5);

    \draw[gray, thick] (5,5.5+-2) -- (5,5.5+2);
    \draw[gray, thick] (1,5.5+-2) -- (1,5.5+2);
    \draw[gray, thick] (5,5.5+2) -- (1,5.5+2);
    \draw[gray, thick] (5,5.5+-2) -- (1,5.5+-2);

    \draw[gray, thick] (-4,5.5+-0.5) -- (-4,5.5+0.5);
    \draw[gray, thick] (-2,5.5+-0.5) -- (-2,5.5+0.5);
    \draw[gray, thick] (-4,5.5+0.5) -- (-2,5.5+0.5);
    \draw[gray, thick] (-4,5.5+-0.5) -- (-2,5.5+-0.5);

    \draw[gray, thick] (-4,5.5+-0.5) -- (1,5.5+-2);
    \draw[gray, thick] (-2,5.5+-0.5) -- (1,5.5+-0.5-0.64285714285);
    \draw[gray, thick, dotted] (1,5.5+-0.5-0.64285714285) -- (5,5.5+-2);
    \draw[gray, thick] (-4,5.5+0.5) -- (1,5.5+2);
    \draw[gray, thick] (-2,5.5+0.5) -- (1,5.5+0.5+0.64285714285);
    \draw[gray, thick,dotted] (1,5.5+0.5+0.64285714285) -- (5,5.5+2);

    
    \filldraw[black] (-3,5.5+0) circle (3/4pt);
    
    \filldraw[gray] (-2,5.5+0) circle (3/4pt);
    \filldraw[gray] (-4,5.5+0) circle (3/4pt);
    \filldraw[black] (-2.5,5.5+0) circle (3/4pt);
    \filldraw[black] (-3.5,5.5+0) circle (3/4pt);
    \filldraw[gray] (-1.5,5.5+0) circle (3/4pt);
    \filldraw[gray] (-4.5,5.5+0) circle (3/4pt);
    \filldraw[black] (-2.25,5.5+0) circle (3/4pt);
    \filldraw[black] (-3.25,5.5+0) circle (3/4pt);
    \filldraw[gray] (-1.25,5.5+0) circle (3/4pt);
    \filldraw[gray] (-4.25,5.5+0) circle (3/4pt);
    \filldraw[black] (-2.75,5.5+0) circle (3/4pt);
    \filldraw[black] (-3.75,5.5+0) circle (3/4pt);
    \filldraw[gray] (-1.75,5.5+0) circle (3/4pt);
    \filldraw[gray] (-4.75,5.5+0) circle (3/4pt);
    \filldraw[black] (-2.25-0.125,5.5+0) circle (3/4pt);
    \filldraw[black] (-3.25-0.125,5.5+0) circle (3/4pt);
    \filldraw[gray] (-1.25-0.125,5.5+0) circle (3/4pt);
    \filldraw[gray] (-4.25-0.125,5.5+0) circle (3/4pt);
    \filldraw[black] (-2.75-0.125,5.5+0) circle (3/4pt);
    \filldraw[black] (-3.75-0.125,5.5+0) circle (3/4pt);
    \filldraw[gray] (-1.75-0.125,5.5+0) circle (3/4pt);
    \filldraw[gray] (-4.75-0.125,5.5+0) circle (3/4pt);
    
    \filldraw[black] (-2.25+0.125,5.5+0) circle (3/4pt);
    \filldraw[black] (-3.25+0.125,5.5+0) circle (3/4pt);
    \filldraw[gray] (-1.25+0.125,5.5+0) circle (3/4pt);
    \filldraw[gray] (-4.25+0.125,5.5+0) circle (3/4pt);
    \filldraw[black] (-2.75+0.125,5.5+0) circle (3/4pt);
    \filldraw[black] (-3.75+0.125,5.5+0) circle (3/4pt);
    \filldraw[gray] (-1.75+0.125,5.5+0) circle (3/4pt);
    \filldraw[gray] (-4.75+0.125,5.5+0) circle (3/4pt);
    \filldraw[black] (-2.25-0.125-0.0625,5.5+0) circle (3/4pt);
    \filldraw[black] (-2.25-0.125+0.0625,5.5+0) circle (3/4pt);
    \filldraw[black] (-3.25-0.125-0.0625,5.5+0) circle (3/4pt);
    \filldraw[black] (-3.25-0.125+0.0625,5.5+0) circle (3/4pt);
    \filldraw[gray] (-1.25-0.125-0.0625,5.5+0) circle (3/4pt);
    \filldraw[gray] (-1.25-0.125+0.0625,5.5+0) circle (3/4pt);
    \filldraw[gray] (-4.25-0.125-0.0625,5.5+0) circle (3/4pt);
    \filldraw[gray] (-4.25-0.125+0.0625,5.5+0) circle (3/4pt);
    \filldraw[black] (-2.75-0.125-0.0625,5.5+0) circle (3/4pt);
    \filldraw[black] (-2.75-0.125+0.0625,5.5+0) circle (3/4pt);
    \filldraw[black] (-3.75-0.125-0.0625,5.5+0) circle (3/4pt);
    \filldraw[black] (-3.75-0.125+0.0625,5.5+0) circle (3/4pt);
    \filldraw[gray] (-1.75-0.125-0.0625,5.5+0) circle (3/4pt);
    \filldraw[gray] (-1.75-0.125+0.0625,5.5+0) circle (3/4pt);
    \filldraw[gray] (-4.75-0.125-0.0625,5.5+0) circle (3/4pt);
    \filldraw[gray] (-4.75-0.125+0.0625,5.5+0) circle (3/4pt);
    
    \filldraw[black] (-2.25+0.125-0.0625,5.5+0) circle (3/4pt);
    \filldraw[black] (-2.25+0.125+0.0625,5.5+0) circle (3/4pt);
    \filldraw[black] (-3.25+0.125-0.0625,5.5+0) circle (3/4pt);
    \filldraw[black] (-3.25+0.125+0.0625,5.5+0) circle (3/4pt);
    \filldraw[gray] (-1.25+0.125-0.0625,5.5+0) circle (3/4pt);
    \filldraw[gray] (-1.25+0.125+0.0625,5.5+0) circle (3/4pt);
    \filldraw[gray] (-4.25+0.125-0.0625,5.5+0) circle (3/4pt);
    \filldraw[gray] (-4.25+0.125+0.0625,5.5+0) circle (3/4pt);
    \filldraw[black] (-2.75+0.125-0.0625,5.5+0) circle (3/4pt);
    \filldraw[black] (-2.75+0.125+0.0625,5.5+0) circle (3/4pt);
    \filldraw[black] (-3.75+0.125-0.0625,5.5+0) circle (3/4pt);
    \filldraw[black] (-3.75+0.125+0.0625,5.5+0) circle (3/4pt);
    \filldraw[gray] (-1.75+0.125-0.0625,5.5+0) circle (3/4pt);
    \filldraw[gray] (-1.75+0.125+0.0625,5.5+0) circle (3/4pt);
    \filldraw[gray] (-4.75+0.125-0.0625,5.5+0) circle (3/4pt);
    \filldraw[gray] (-4.75+0.125+0.0625,5.5+0) circle (3/4pt);

    \filldraw[gray] (-3,5.5+1) circle (3/4pt);
    \filldraw[gray] (-3,5.5+-1) circle (3/4pt);
    \filldraw[gray] (-3,5.5+-0.5) circle (3/4pt);
    \filldraw[gray] (-3,5.5+-1.5) circle (3/4pt);
    \filldraw[gray] (-3,5.5+0.5) circle (3/4pt);
    \filldraw[gray] (-3,5.5+1.5) circle (3/4pt);
    \filldraw[black] (-3,5.5+-0.25) circle (3/4pt);
    \filldraw[black] (-3,5.5+0.25) circle (3/4pt);
    \filldraw[gray] (-3,5.5+1.25) circle (3/4pt);
    \filldraw[gray] (-3,5.5+-1.25) circle (3/4pt);
    \filldraw[gray] (-3,5.5+0.75) circle (3/4pt);
    \filldraw[gray] (-3,5.5+-0.75) circle (3/4pt);
    \filldraw[gray] (-3,5.5+-1.75) circle (3/4pt);
    \filldraw[gray] (-3,5.5+1.75) circle (3/4pt);
    \filldraw[black] (-3,5.5+-0.25-0.125) circle (3/4pt);
    \filldraw[black] (-3,5.5+0.25-0.125) circle (3/4pt);
    \filldraw[gray] (-3,5.5+1.25-0.125) circle (3/4pt);
    \filldraw[gray] (-3,5.5+-1.25-0.125) circle (3/4pt);
    \filldraw[gray] (-3,5.5+0.75-0.125) circle (3/4pt);
    \filldraw[gray] (-3,5.5+-0.75-0.125) circle (3/4pt);
    \filldraw[gray] (-3,5.5+-1.75-0.125) circle (3/4pt);
    \filldraw[gray] (-3,5.5+1.75-0.125) circle (3/4pt);
    
    \filldraw[black] (-3,5.5+-0.25+0.125) circle (3/4pt);
    \filldraw[black] (-3,5.5+0.25+0.125) circle (3/4pt);
    \filldraw[gray] (-3,5.5+1.25+0.125) circle (3/4pt);
    \filldraw[gray] (-3,5.5+-1.25+0.125) circle (3/4pt);
    \filldraw[gray] (-3,5.5+0.75+0.125) circle (3/4pt);
    \filldraw[gray] (-3,5.5+-0.75+0.125) circle (3/4pt);
    \filldraw[gray] (-3,5.5+-1.75+0.125) circle (3/4pt);
    \filldraw[gray] (-3,5.5+1.75+0.125) circle (3/4pt);
    \filldraw[black] (-3,5.5+-0.25-0.125-0.0625) circle (3/4pt);
    \filldraw[black] (-3,5.5+-0.25-0.125+0.0625) circle (3/4pt);
    \filldraw[black] (-3,5.5+0.25-0.125-0.0625) circle (3/4pt);
    \filldraw[black] (-3,5.5+0.25-0.125+0.0625) circle (3/4pt);
    \filldraw[gray] (-3,5.5+1.25-0.125-0.0625) circle (3/4pt);
    \filldraw[gray] (-3,5.5+1.25-0.125+0.0625) circle (3/4pt);
    \filldraw[gray] (-3,5.5+-1.25-0.125-0.0625) circle (3/4pt);
    \filldraw[gray] (-3,5.5+-1.25-0.125+0.0625) circle (3/4pt);
    \filldraw[gray] (-3,5.5+0.75-0.125-0.0625) circle (3/4pt);
    \filldraw[gray] (-3,5.5+0.75-0.125+0.0625) circle (3/4pt);
    \filldraw[gray] (-3,5.5+-0.75-0.125-0.0625) circle (3/4pt);
    \filldraw[gray] (-3,5.5+-0.75-0.125+0.0625) circle (3/4pt);
    \filldraw[gray] (-3,5.5+-1.75-0.125-0.0625) circle (3/4pt);
    \filldraw[gray] (-3,5.5+-1.75-0.125+0.0625) circle (3/4pt);
    \filldraw[gray] (-3,5.5+1.75-0.125-0.0625) circle (3/4pt);
    \filldraw[gray] (-3,5.5+1.75-0.125+0.0625) circle (3/4pt);
    
    \filldraw[black] (-3,5.5+-0.25+0.125-0.0625) circle (3/4pt);
    \filldraw[black] (-3,5.5+-0.25+0.125+0.0625) circle (3/4pt);
    \filldraw[black] (-3,5.5+0.25+0.125-0.0625) circle (3/4pt);
    \filldraw[black] (-3,5.5+0.25+0.125+0.0625) circle (3/4pt);
    \filldraw[gray] (-3,5.5+1.25+0.125-0.0625) circle (3/4pt);
    \filldraw[gray] (-3,5.5+1.25+0.125+0.0625) circle (3/4pt);
    \filldraw[gray] (-3,5.5+-1.25+0.125-0.0625) circle (3/4pt);
    \filldraw[gray] (-3,5.5+-1.25+0.125+0.0625) circle (3/4pt);
    \filldraw[gray] (-3,5.5+0.75+0.125-0.0625) circle (3/4pt);
    \filldraw[gray] (-3,5.5+0.75+0.125+0.0625) circle (3/4pt);
    \filldraw[gray] (-3,5.5+-0.75+0.125-0.0625) circle (3/4pt);
    \filldraw[gray] (-3,5.5+-0.75+0.125+0.0625) circle (3/4pt);
    \filldraw[gray] (-3,5.5+-1.75+0.125-0.0625) circle (3/4pt);
    \filldraw[gray] (-3,5.5+-1.75+0.125+0.0625) circle (3/4pt);
    \filldraw[gray] (-3,5.5+1.75+0.125-0.0625) circle (3/4pt);
    \filldraw[gray] (-3,5.5+1.75+0.125+0.0625) circle (3/4pt);

    \filldraw[gray] (-4,5.5+1) circle (3/4pt);
    \filldraw[gray] (-2,5.5+1) circle (3/4pt);
    \filldraw[gray] (-2,5.5+-1) circle (3/4pt);
    \filldraw[gray] (-4,5.5+-1) circle (3/4pt);

    \filldraw[gray] (1/2 + -4,5.5+0 + 1) circle (3/4pt);
    \filldraw[gray] (-1/2 + -4,5.5+0 + 1) circle (3/4pt);
    \filldraw[gray] (-0.5/2 + -4,5.5+0 + 1) circle (3/4pt);
    \filldraw[gray] (-1.5/2 + -4,5.5+0 + 1) circle (3/4pt);
    \filldraw[gray] (1.5/2 + -4,5.5+0 + 1) circle (3/4pt);
    \filldraw[gray] (0.5/2 + -4,5.5+0 + 1) circle (3/4pt);
    \filldraw[gray] (-0.5/2-0.5/4 + -4,5.5+0 + 1) circle (3/4pt);
    \filldraw[gray] (-1.5/2-0.5/4 + -4,5.5+0 + 1) circle (3/4pt);
    \filldraw[gray] (1.5/2-0.5/4 + -4,5.5+0 + 1) circle (3/4pt);
    \filldraw[gray] (0.5/2-0.5/4 + -4,5.5+0 + 1) circle (3/4pt);

    \filldraw[gray] (-0.5/2+0.5/4 + -4,5.5+0 + 1) circle (3/4pt);
    \filldraw[gray] (-1.5/2+0.5/4 + -4,5.5+0 + 1) circle (3/4pt);
    \filldraw[gray] (1.5/2+0.5/4 + -4,5.5+0 + 1) circle (3/4pt);
    \filldraw[gray] (0.5/2+0.5/4 + -4,5.5+0 + 1) circle (3/4pt);
    \filldraw[gray] (0 + -4,5.5+1/2 + 1) circle (3/4pt);
    \filldraw[gray] (0 + -4,5.5+-1/2 + 1) circle (3/4pt);
    \filldraw[gray] (0 + -4,5.5+-0.5/2 + 1) circle (3/4pt);
    \filldraw[gray] (0 + -4,5.5+-1.5/2 + 1) circle (3/4pt);
    \filldraw[gray] (0 + -4,5.5+1.5/2 + 1) circle (3/4pt);
    \filldraw[gray] (0 + -4,5.5+0.5/2 + 1) circle (3/4pt);
    \filldraw[gray] (0 + -4,5.5+-0.5/2-0.5/4 + 1) circle (3/4pt);
    \filldraw[gray] (0 + -4,5.5+-1.5/2-0.5/4 + 1) circle (3/4pt);
    \filldraw[gray] (0 + -4,5.5+1.5/2-0.5/4 + 1) circle (3/4pt);
    \filldraw[gray] (0 + -4,5.5+0.5/2-0.5/4 + 1) circle (3/4pt);
    
    \filldraw[gray] (0 + -4,5.5+-0.5/2+0.5/4 + 1) circle (3/4pt);
    \filldraw[gray] (0 + -4,5.5+-1.5/2+0.5/4 + 1) circle (3/4pt);
    \filldraw[gray] (0 + -4,5.5+1.5/2+0.5/4 + 1) circle (3/4pt);
    \filldraw[gray] (0 + -4,5.5+0.5/2+0.5/4 + 1) circle (3/4pt);

    \filldraw[gray] (1/2 + -2,5.5+0 + 1) circle (3/4pt);
    \filldraw[gray] (-1/2 + -2,5.5+0 + 1) circle (3/4pt);
    \filldraw[gray] (-0.5/2 + -2,5.5+0 + 1) circle (3/4pt);
    \filldraw[gray] (-1.5/2 + -2,5.5+0 + 1) circle (3/4pt);
    \filldraw[gray] (1.5/2 + -2,5.5+0 + 1) circle (3/4pt);
    \filldraw[gray] (0.5/2 + -2,5.5+0 + 1) circle (3/4pt);
    \filldraw[gray] (-0.5/2+0.5/4 + -2,5.5+0 + 1) circle (3/4pt);
    \filldraw[gray] (-1.5/2+0.5/4 + -2,5.5+0 + 1) circle (3/4pt);
    \filldraw[gray] (1.5/2+0.5/4 + -2,5.5+0 + 1) circle (3/4pt);
    \filldraw[gray] (0.5/2+0.5/4 + -2,5.5+0 + 1) circle (3/4pt);

    \filldraw[gray] (-0.5/2-0.5/4 + -2,5.5+0 + 1) circle (3/4pt);
    \filldraw[gray] (-1.5/2-0.5/4 + -2,5.5+0 + 1) circle (3/4pt);
    \filldraw[gray] (1.5/2-0.5/4 + -2,5.5+0 + 1) circle (3/4pt);
    \filldraw[gray] (0.5/2-0.5/4 + -2,5.5+0 + 1) circle (3/4pt);
    \filldraw[gray] (0 + -2,5.5+1/2 + 1) circle (3/4pt);
    \filldraw[gray] (0 + -2,5.5+-1/2 + 1) circle (3/4pt);
    \filldraw[gray] (0 + -2,5.5+-0.5/2 + 1) circle (3/4pt);
    \filldraw[gray] (0 + -2,5.5+-1.5/2 + 1) circle (3/4pt);
    \filldraw[gray] (0 + -2,5.5+1.5/2 + 1) circle (3/4pt);
    \filldraw[gray] (0 + -2,5.5+0.5/2 + 1) circle (3/4pt);
    \filldraw[gray] (0 + -2,5.5+-0.5/2-0.5/4 + 1) circle (3/4pt);
    \filldraw[gray] (0 + -2,5.5+-1.5/2-0.5/4 + 1) circle (3/4pt);
    \filldraw[gray] (0 + -2,5.5+1.5/2-0.5/4 + 1) circle (3/4pt);
    \filldraw[gray] (0 + -2,5.5+0.5/2-0.5/4 + 1) circle (3/4pt);

    \filldraw[gray] (0 + -2,5.5+-0.5/2+0.5/4 + 1) circle (3/4pt);
    \filldraw[gray] (0 + -2,5.5+-1.5/2+0.5/4 + 1) circle (3/4pt);
    \filldraw[gray] (0 + -2,5.5+1.5/2+0.5/4 + 1) circle (3/4pt);
    \filldraw[gray] (0 + -2,5.5+0.5/2+0.5/4 + 1) circle (3/4pt);

    \filldraw[gray] (1/2 + -2,5.5+0 - 1) circle (3/4pt);
    \filldraw[gray] (-1/2 + -2,5.5+0 - 1) circle (3/4pt);
    \filldraw[gray] (-0.5/2 + -2,5.5+0 - 1) circle (3/4pt);
    \filldraw[gray] (-1.5/2 + -2,5.5+0 - 1) circle (3/4pt);
    \filldraw[gray] (1.5/2 + -2,5.5+0 - 1) circle (3/4pt);
    \filldraw[gray] (0.5/2 + -2,5.5+0 - 1) circle (3/4pt);
    \filldraw[gray] (-0.5/2+0.5/4 + -2,5.5+0 - 1) circle (3/4pt);
    \filldraw[gray] (-1.5/2+0.5/4 + -2,5.5+0 - 1) circle (3/4pt);
    \filldraw[gray] (1.5/2+0.5/4 + -2,5.5+0 - 1) circle (3/4pt);
    \filldraw[gray] (0.5/2+0.5/4 + -2,5.5+0 - 1) circle (3/4pt);

    \filldraw[gray] (-0.5/2-0.5/4 + -2,5.5+0 - 1) circle (3/4pt);
    \filldraw[gray] (-1.5/2-0.5/4 + -2,5.5+0 - 1) circle (3/4pt);
    \filldraw[gray] (1.5/2-0.5/4 + -2,5.5+0 - 1) circle (3/4pt);
    \filldraw[gray] (0.5/2-0.5/4 + -2,5.5+0 - 1) circle (3/4pt);
    \filldraw[gray] (0 + -2,5.5+1/2 - 1) circle (3/4pt);
    \filldraw[gray] (0 + -2,5.5+-1/2 - 1) circle (3/4pt);
    \filldraw[gray] (0 + -2,5.5+-0.5/2 - 1) circle (3/4pt);
    \filldraw[gray] (0 + -2,5.5+-1.5/2 - 1) circle (3/4pt);
    \filldraw[gray] (0 + -2,5.5+1.5/2 - 1) circle (3/4pt);
    \filldraw[gray] (0 + -2,5.5+0.5/2 - 1) circle (3/4pt);
    \filldraw[gray] (0 + -2,5.5+-0.5/2+0.5/4 - 1) circle (3/4pt);
    \filldraw[gray] (0 + -2,5.5+-1.5/2+0.5/4 - 1) circle (3/4pt);
    \filldraw[gray] (0 + -2,5.5+1.5/2+0.5/4 - 1) circle (3/4pt);
    \filldraw[gray] (0 + -2,5.5+0.5/2+0.5/4 - 1) circle (3/4pt);
    
    \filldraw[gray] (0 + -2,5.5+-0.5/2-0.5/4 - 1) circle (3/4pt);
    \filldraw[gray] (0 + -2,5.5+-1.5/2-0.5/4 - 1) circle (3/4pt);
    \filldraw[gray] (0 + -2,5.5+1.5/2-0.5/4 - 1) circle (3/4pt);
    \filldraw[gray] (0 + -2,5.5+0.5/2-0.5/4 - 1) circle (3/4pt);

    \filldraw[gray] (1/2 + -4,5.5+0 - 1) circle (3/4pt);
    \filldraw[gray] (-1/2 + -4,5.5+0 - 1) circle (3/4pt);
    \filldraw[gray] (-0.5/2 + -4,5.5+0 - 1) circle (3/4pt);
    \filldraw[gray] (-1.5/2 + -4,5.5+0 - 1) circle (3/4pt);
    \filldraw[gray] (1.5/2 + -4,5.5+0 - 1) circle (3/4pt);
    \filldraw[gray] (0.5/2 + -4,5.5+0 - 1) circle (3/4pt);
    \filldraw[gray] (-0.5/2-0.5/4 + -4,5.5+0 - 1) circle (3/4pt);
    \filldraw[gray] (-1.5/2-0.5/4 + -4,5.5+0 - 1) circle (3/4pt);
    \filldraw[gray] (1.5/2-0.5/4 + -4,5.5+0 - 1) circle (3/4pt);
    \filldraw[gray] (0.5/2-0.5/4 + -4,5.5+0 - 1) circle (3/4pt);

    \filldraw[gray] (-0.5/2+0.5/4 + -4,5.5+0 - 1) circle (3/4pt);
    \filldraw[gray] (-1.5/2+0.5/4 + -4,5.5+0 - 1) circle (3/4pt);
    \filldraw[gray] (1.5/2+0.5/4 + -4,5.5+0 - 1) circle (3/4pt);
    \filldraw[gray] (0.5/2+0.5/4 + -4,5.5+0 - 1) circle (3/4pt);
    \filldraw[gray] (0 + -4,5.5+1/2 - 1) circle (3/4pt);
    \filldraw[gray] (0 + -4,5.5+-1/2 - 1) circle (3/4pt);
    \filldraw[gray] (0 + -4,5.5+-0.5/2 - 1) circle (3/4pt);
    \filldraw[gray] (0 + -4,5.5+-1.5/2 - 1) circle (3/4pt);
    \filldraw[gray] (0 + -4,5.5+1.5/2 - 1) circle (3/4pt);
    \filldraw[gray] (0 + -4,5.5+0.5/2 - 1) circle (3/4pt);
    \filldraw[gray] (0 + -4,5.5+-0.5/2+0.5/4 - 1) circle (3/4pt);
    \filldraw[gray] (0 + -4,5.5+-1.5/2+0.5/4 - 1) circle (3/4pt);
    \filldraw[gray] (0 + -4,5.5+1.5/2+0.5/4 - 1) circle (3/4pt);
    \filldraw[gray] (0 + -4,5.5+0.5/2+0.5/4 - 1) circle (3/4pt);

    \filldraw[gray] (0 + -4,5.5+-0.5/2-0.5/4 - 1) circle (3/4pt);
    \filldraw[gray] (0 + -4,5.5+-1.5/2-0.5/4 - 1) circle (3/4pt);
    \filldraw[gray] (0 + -4,5.5+1.5/2-0.5/4 - 1) circle (3/4pt);
    \filldraw[gray] (0 + -4,5.5+0.5/2-0.5/4 - 1) circle (3/4pt);


    \filldraw[gray] (-4 + 0.5,5.5+1+ 0.5) circle (3/4pt);
    \filldraw[gray] (-2 + 0.5,5.5+1 + 0.5) circle (3/4pt);
    \filldraw[gray] (-2 + 0.5,5.5+-1 + 0.5) circle (3/4pt);
    \filldraw[gray] (-4 + 0.5,5.5+-1 + 0.5) circle (3/4pt);

    \filldraw[gray] (-4 + 0.5,5.5+1 - 0.5) circle (3/4pt);
    \filldraw[gray] (-2 + 0.5,5.5+1 - 0.5) circle (3/4pt);
    \filldraw[gray] (-2 + 0.5,5.5+-1 - 0.5) circle (3/4pt);
    \filldraw[gray] (-4 + 0.5,5.5+-1 - 0.5) circle (3/4pt);

    \filldraw[gray] (-4 - 0.5,5.5+1 + 0.5) circle (3/4pt);
    \filldraw[gray] (-2 - 0.5,5.5+1 + 0.5) circle (3/4pt);
    \filldraw[gray] (-2 - 0.5,5.5+-1 + 0.5) circle (3/4pt);
    \filldraw[gray] (-4 - 0.5,5.5+-1 + 0.5) circle (3/4pt);
    
    \filldraw[gray] (-4 - 0.5,5.5+1 - 0.5) circle (3/4pt);
    \filldraw[gray] (-2 - 0.5,5.5+1 - 0.5) circle (3/4pt);
    \filldraw[gray] (-2 - 0.5,5.5+-1 - 0.5) circle (3/4pt);
    \filldraw[gray] (-4 - 0.5,5.5+-1 - 0.5) circle (3/4pt);

    \filldraw[gray] (-4 + 0.5 + 0.5/2,5.5+1+ 0.5) circle (3/4pt);
    \filldraw[gray] (-2 + 0.5 + 0.5/2,5.5+1 + 0.5) circle (3/4pt);
    \filldraw[gray] (-2 + 0.5 + 0.5/2,5.5+-1 + 0.5) circle (3/4pt);
    \filldraw[gray] (-4 + 0.5 + 0.5/2,5.5+-1 + 0.5) circle (3/4pt);

    \filldraw[gray] (-4 + 0.5 + 0.5/2,5.5+1 - 0.5) circle (3/4pt);
    \filldraw[gray] (-2 + 0.5 + 0.5/2,5.5+1 - 0.5) circle (3/4pt);
    \filldraw[gray] (-2 + 0.5 + 0.5/2,5.5+-1 - 0.5) circle (3/4pt);
    \filldraw[gray] (-4 + 0.5 + 0.5/2,5.5+-1 - 0.5) circle (3/4pt);

    \filldraw[gray] (-4 - 0.5 + 0.5/2,5.5+1 + 0.5) circle (3/4pt);
    \filldraw[gray] (-2 - 0.5 + 0.5/2,5.5+1 + 0.5) circle (3/4pt);
    \filldraw[gray] (-2 - 0.5 + 0.5/2,5.5+-1 + 0.5) circle (3/4pt);
    \filldraw[gray] (-4 - 0.5 + 0.5/2,5.5+-1 + 0.5) circle (3/4pt);
    
    \filldraw[gray] (-4 - 0.5 + 0.5/2,5.5+1 - 0.5) circle (3/4pt);
    \filldraw[gray] (-2 - 0.5 + 0.5/2,5.5+1 - 0.5) circle (3/4pt);
    \filldraw[gray] (-2 - 0.5 + 0.5/2,5.5+-1 - 0.5) circle (3/4pt);
    \filldraw[gray] (-4 - 0.5 + 0.5/2,5.5+-1 - 0.5) circle (3/4pt);
    \filldraw[gray] (-4 + 0.5 - 0.5/2,5.5+1+ 0.5) circle (3/4pt);
    \filldraw[gray] (-2 + 0.5 - 0.5/2,5.5+1 + 0.5) circle (3/4pt);
    \filldraw[gray] (-2 + 0.5 - 0.5/2,5.5+-1 + 0.5) circle (3/4pt);
    \filldraw[gray] (-4 + 0.5 - 0.5/2,5.5+-1 + 0.5) circle (3/4pt);

    \filldraw[gray] (-4 + 0.5 - 0.5/2,5.5+1 - 0.5) circle (3/4pt);
    \filldraw[gray] (-2 + 0.5 - 0.5/2,5.5+1 - 0.5) circle (3/4pt);
    \filldraw[gray] (-2 + 0.5 - 0.5/2,5.5+-1 - 0.5) circle (3/4pt);
    \filldraw[gray] (-4 + 0.5 - 0.5/2,5.5+-1 - 0.5) circle (3/4pt);

    \filldraw[gray] (-4 - 0.5 - 0.5/2,5.5+1 + 0.5) circle (3/4pt);
    \filldraw[gray] (-2 - 0.5 - 0.5/2,5.5+1 + 0.5) circle (3/4pt);
    \filldraw[gray] (-2 - 0.5 - 0.5/2,5.5+-1 + 0.5) circle (3/4pt);
    \filldraw[gray] (-4 - 0.5 - 0.5/2,5.5+-1 + 0.5) circle (3/4pt);
    
    \filldraw[gray] (-4 - 0.5 - 0.5/2,5.5+1 - 0.5) circle (3/4pt);
    \filldraw[gray] (-2 - 0.5 - 0.5/2,5.5+1 - 0.5) circle (3/4pt);
    \filldraw[gray] (-2 - 0.5 - 0.5/2,5.5+-1 - 0.5) circle (3/4pt);
    \filldraw[gray] (-4 - 0.5 - 0.5/2,5.5+-1 - 0.5) circle (3/4pt);

    \filldraw[gray] (-4 + 0.5,5.5+1+ 0.5 + 0.5/2) circle (3/4pt);
    \filldraw[gray] (-2 + 0.5,5.5+1 + 0.5 + 0.5/2) circle (3/4pt);
    \filldraw[gray] (-2 + 0.5,5.5+-1 + 0.5 + 0.5/2) circle (3/4pt);
    \filldraw[black] (-4 + 0.5,5.5+-1 + 0.5 + 0.5/2) circle (3/4pt);

    \filldraw[gray] (-4 + 0.5,5.5+1 - 0.5 + 0.5/2) circle (3/4pt);
    \filldraw[gray] (-2 + 0.5,5.5+1 - 0.5 + 0.5/2) circle (3/4pt);
    \filldraw[gray] (-2 + 0.5,5.5+-1 - 0.5 + 0.5/2) circle (3/4pt);
    \filldraw[gray] (-4 + 0.5,5.5+-1 - 0.5 + 0.5/2) circle (3/4pt);

    \filldraw[gray] (-4 - 0.5,5.5+1 + 0.5 + 0.5/2) circle (3/4pt);
    \filldraw[gray] (-2 - 0.5,5.5+1 + 0.5 + 0.5/2) circle (3/4pt);
    \filldraw[black] (-2 - 0.5,5.5+-1 + 0.5 + 0.5/2) circle (3/4pt);
    \filldraw[gray] (-4 - 0.5,5.5+-1 + 0.5 + 0.5/2) circle (3/4pt);
    
    \filldraw[gray] (-4 - 0.5,5.5+1 - 0.5 + 0.5/2) circle (3/4pt);
    \filldraw[gray] (-2 - 0.5,5.5+1 - 0.5 + 0.5/2) circle (3/4pt);
    \filldraw[gray] (-2 - 0.5,5.5+-1 - 0.5 + 0.5/2) circle (3/4pt);
    \filldraw[gray] (-4 - 0.5,5.5+-1 - 0.5 + 0.5/2) circle (3/4pt);
    \filldraw[gray] (-4 + 0.5,5.5+1+ 0.5 - 0.5/2) circle (3/4pt);
    \filldraw[gray] (-2 + 0.5,5.5+1 + 0.5 - 0.5/2) circle (3/4pt);
    \filldraw[gray] (-2 + 0.5,5.5+-1 + 0.5 - 0.5/2) circle (3/4pt);
    \filldraw[gray] (-4 + 0.5,5.5+-1 + 0.5 - 0.5/2) circle (3/4pt);

    \filldraw[black] (-4 + 0.5,5.5+1 - 0.5 - 0.5/2) circle (3/4pt);
    \filldraw[gray] (-2 + 0.5,5.5+1 - 0.5 - 0.5/2) circle (3/4pt);
    \filldraw[gray] (-2 + 0.5,5.5+-1 - 0.5 - 0.5/2) circle (3/4pt);
    \filldraw[gray] (-4 + 0.5,5.5+-1 - 0.5 - 0.5/2) circle (3/4pt);

    \filldraw[gray] (-4 - 0.5,5.5+1 + 0.5 - 0.5/2) circle (3/4pt);
    \filldraw[gray] (-2 - 0.5,5.5+1 + 0.5 - 0.5/2) circle (3/4pt);
    \filldraw[gray] (-2 - 0.5,5.5+-1 + 0.5 - 0.5/2) circle (3/4pt);
    \filldraw[gray] (-4 - 0.5,5.5+-1 + 0.5 - 0.5/2) circle (3/4pt);
    
    \filldraw[gray] (-4 - 0.5,5.5+1 - 0.5 - 0.5/2) circle (3/4pt);
    \filldraw[black] (-2 - 0.5,5.5+1 - 0.5 - 0.5/2) circle (3/4pt);
    \filldraw[gray] (-2 - 0.5,5.5+-1 - 0.5 - 0.5/2) circle (3/4pt);
    \filldraw[gray] (-4 - 0.5,5.5+-1 - 0.5 - 0.5/2) circle (3/4pt);

    \filldraw[black] (3,5.5+0) circle (3/4pt);
    
    \filldraw[black] (2,5.5+0) circle (3/4pt);
    \filldraw[black] (4,5.5+0) circle (3/4pt);
    \filldraw[black] (2.5,5.5+0) circle (3/4pt);
    \filldraw[black] (3.5,5.5+0) circle (3/4pt);
    \filldraw[black] (1.5,5.5+0) circle (3/4pt);
    \filldraw[black] (4.5,5.5+0) circle (3/4pt);
    \filldraw[black] (2.25,5.5+0) circle (3/4pt);
    \filldraw[black] (3.25,5.5+0) circle (3/4pt);
    \filldraw[black] (1.25,5.5+0) circle (3/4pt);
    \filldraw[black] (4.25,5.5+0) circle (3/4pt);
    \filldraw[black] (2.75,5.5+0) circle (3/4pt);
    \filldraw[black] (3.75,5.5+0) circle (3/4pt);
    \filldraw[black] (1.75,5.5+0) circle (3/4pt);
    \filldraw[black] (4.75,5.5+0) circle (3/4pt);
    \filldraw[black] (2.25+0.5/4,5.5+0) circle (3/4pt);
    \filldraw[black] (3.25+0.5/4,5.5+0) circle (3/4pt);
    \filldraw[black] (1.25+0.5/4,5.5+0) circle (3/4pt);
    \filldraw[black] (4.25+0.5/4,5.5+0) circle (3/4pt);
    \filldraw[black] (2.75+0.5/4,5.5+0) circle (3/4pt);
    \filldraw[black] (3.75+0.5/4,5.5+0) circle (3/4pt);
    \filldraw[black] (1.75+0.5/4,5.5+0) circle (3/4pt);
    \filldraw[black] (4.75+0.5/4,5.5+0) circle (3/4pt);
    
    \filldraw[black] (2.25-0.5/4,5.5+0) circle (3/4pt);
    \filldraw[black] (3.25-0.5/4,5.5+0) circle (3/4pt);
    \filldraw[black] (1.25-0.5/4,5.5+0) circle (3/4pt);
    \filldraw[black] (4.25-0.5/4,5.5+0) circle (3/4pt);
    \filldraw[black] (2.75-0.5/4,5.5+0) circle (3/4pt);
    \filldraw[black] (3.75-0.5/4,5.5+0) circle (3/4pt);
    \filldraw[black] (1.75-0.5/4,5.5+0) circle (3/4pt);
    \filldraw[black] (4.75-0.5/4,5.5+0) circle (3/4pt);

    \filldraw[black] (3,5.5+1) circle (3/4pt);
    \filldraw[black] (3,5.5+-1) circle (3/4pt);
    \filldraw[black] (3,5.5+-0.5) circle (3/4pt);
    \filldraw[black] (3,5.5+-1.5) circle (3/4pt);
    \filldraw[black] (3,5.5+0.5) circle (3/4pt);
    \filldraw[black] (3,5.5+1.5) circle (3/4pt);
    \filldraw[black] (3,5.5+-0.5+0.5/2) circle (3/4pt);
    \filldraw[black] (3,5.5+-1.5+0.5/2) circle (3/4pt);
    \filldraw[black] (3,5.5+0.5+0.5/2) circle (3/4pt);
    \filldraw[black] (3,5.5+1.5+0.5/2) circle (3/4pt);
    \filldraw[black] (3,5.5+-0.5-0.5/2) circle (3/4pt);
    \filldraw[black] (3,5.5+-1.5-0.5/2) circle (3/4pt);
    \filldraw[black] (3,5.5+0.5-0.5/2) circle (3/4pt);
    \filldraw[black] (3,5.5+1.5-0.5/2) circle (3/4pt);

    \filldraw[black] (4,5.5+1) circle (3/4pt);
    \filldraw[black] (2,5.5+-1) circle (3/4pt);
    \filldraw[black] (2,5.5+1) circle (3/4pt);
    \filldraw[black] (4,5.5-1) circle (3/4pt);


    \node[black] at (-5,5.5+-2.3) {\small-0.5};
    \node[black] at (-1,5.5+-2.3) {\small0.5};
    \node[black] at (5,5.5+-2.3) {\small0.5};
    \node[black] at (1,5.5+-2.3) {\small-0.5};

    \node[black] at (-5.5,5.5+-2) {\small-0.5};
    \node[black] at (-5.5,5.5+2) {\small0.5};
    \node[black] at (5.5,5.5+-2) {\small-0.5};
    \node[black] at (5.5,5.5+2) {\small0.5};

    \node[black] at (-3,5.5+-2.5) {\Large\(x_1\)};
    \node[black] at (-5.5,5.5+0) {\Large\(x_2\)};
    \node[black] at (3,5.5+-2.5) {\Large\(x_1\)};
    \node[black] at (5.5,5.5+0) {\Large\(x_2\)};

    \node[black] at (-3,8.8) {Isotropic sparse grid};
    \node[black] at (-3,8.2) {\(\mathbf{p}=(0,0)\)};
    \node[black] at (3,8.8) {Lengthscale-informed};
    \node at (3,8.2) {\(\mathbf{p}=(1,2)\)};


\end{tikzpicture}
\caption{Isotropic and lengthscale-informed sparse grids in two dimensions of level \(L=5\), each corresponding to penalties \(\mathbf{p}=(0,0)\) and \(\mathbf{p}=(1,2)\), respectively.}\label{fig: ISG and LISG}
\end{figure}
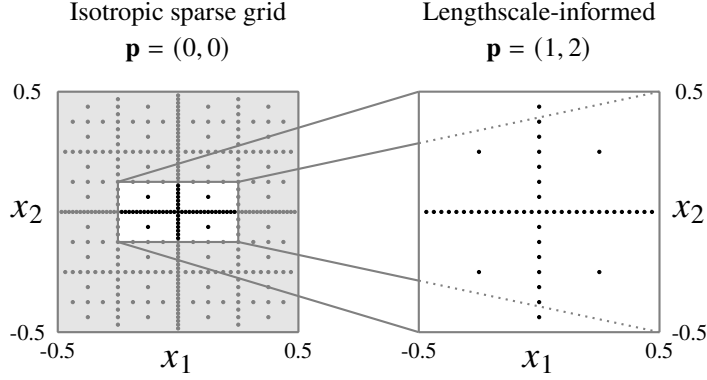

In \cite{addy2025lengthscaleinformedsparsegridskernel}, LISGs were developed to exploit \emph{lengthscale anisotropy}. A function \(f\in H_{\textnormal{mix}}^{\boldsymbol{\nu}+\mathbf{1}/2}(\Gamma^d)\) is considered to exhibit lengthscale anisotropy according to \(\boldsymbol{\lambda}\in\mathbb{R}_{>0}^d\) if
\begin{align}
    \left\lVert f \right\rVert_{\mathcal{N}_{\Phi_{\boldsymbol{\nu},\boldsymbol{\lambda}}}(\Gamma^d)}\leq C_{f}^{(\boldsymbol{\nu},\boldsymbol{\lambda})},\label{eq: lengthscale anisotropy assumption}
\end{align}
for some sufficiently small \(C_{f}^{(\boldsymbol{\nu},\boldsymbol{\lambda})}\in\mathbb{R}_{>0}\). In contexts where the input dimension of \(f\) may be chosen to be arbitrarily large---often by truncating an underlying infinite-dimensional parameter space to meet a prescribed accuracy---it is particularly desirable to identify a sequence \(\{\lambda_j\}_{j\geq1}\) such that the constant \(C_{f}^{(\boldsymbol{\nu},\boldsymbol{\lambda})}\) remains bounded independently of \(d\). 

Employing the isotropic multi-index set \(\mathcal{I}_L\), we obtain LISG interpolation by choosing kernel hyperparameters and the penalty vector \(\mathbf{p}\in\mathbb{N}_0^d\) to reflect the assumed lengthscale-anisotropy of \(f\), according to the relation \(\lambda_j=2^{p_j}\). In Figure~\ref{fig: ISG and LISG}, a two-dimensional LISG of penalty \(\mathbf{p}=(1,2)\) is illustrated in contrast to an ISG of the same level. We have the following bound for the approximation error, where we define the set of subsets of \(\{1,\dots,d\}\) of size \(k\) by \(\mathcal{P}^d_k\coloneqq\{\mathfrak{u}\subset\{1,\dots,d\}\,:\,|\mathfrak{u}|=k, \mathfrak{u}_i<\mathfrak{u}_{i+1}\}\) and, for any \(\mathbf{v}\in\mathbb{R}^d\) and \(\mathfrak{u}\subset\mathcal{P}_k^d\), we denote \((v_{\mathfrak{u}_1},\dots,v_{\mathfrak{u}_k})\in\mathbb{N}_0^k\) by \(\mathbf{v}_{\mathfrak{u}}\).

\begin{theorem}[Theorem 2, \cite{addy2025lengthscaleinformedsparsegridskernel}]\label{thm: error in L}
    Let \(\mathbf{p}\in\mathbb{N}_0^d\) and \(L\in\mathbb{N}_0\) be given, and let \(\boldsymbol{\nu},\boldsymbol{\alpha}\in\mathbb{R}^d_{\geq1/2}\) be such that \(\nu_j-\alpha_j=c\in\mathbb{R}_{\geq0}\) for all \(1\leq j \leq d\). For a constant \(C^{(\boldsymbol{\nu},\boldsymbol{\alpha})}_{\textnormal{L}\ref{lem: tensor hilbert}}\) independent of \(L\) and \(\mathbf{p}\) and defined in Lemma~\ref{lem: tensor hilbert}, we have
    \begin{align}
    \begin{split}
    \left\lVert\, I-S_{\mathcal{I}_L,\mathbf{p}}^{(\boldsymbol{\nu},2^\mathbf{p})}\,\right\rVert_{\mathcal{N}_{\Phi_{\boldsymbol{\nu},2^{\mathbf{p}}}}(\Gamma^d)\rightarrow\mathcal{N}_{\Phi_{\boldsymbol{\alpha},2^{\mathbf{p}}}}(\Gamma^d)}
    \leq C^{(\boldsymbol{\nu},\boldsymbol{\alpha})}_{\textnormal{L}\ref{lem: tensor hilbert}}\sum_{k=1}^d2^{-ck}\sum_{\mathfrak{u}\in\mathcal{P}^d_k}2^{-c|\mathbf{p}_{\mathfrak{u}}|_1}\epsilon^{(k)}_{\boldsymbol{\nu}_{\mathfrak{u}},\boldsymbol{\alpha}_{\mathfrak
    {u}}}(L-|\mathbf{p}_{\mathfrak{u}}|_1-k),
    \end{split}\nonumber
    \end{align}
    where \(\epsilon^{(k)}_{\boldsymbol{\nu}_{\mathfrak{u}},\boldsymbol{\alpha}_{\mathfrak{u}}}(L)\) is 
    as defined in Theorem~\ref{thm: nobile result}.
\end{theorem}

In contrast to Theorem~\ref{thm: nobile result}, the bound is expressed directly in the Native space norm associated with the separable Mat\'ern kernels, rather than in standard mixed Sobolev norms. By Proposition~\ref{prop: matern native spaces}, these norms are equivalent. Additionally, Proposition~7 in \cite{addy2025lengthscaleinformedsparsegridskernel} immediately provides an upper bound for the \(L^\infty\)-error via the relationship \(\|\cdot\|_{L^\infty(\Gamma^d)}\leq|\boldsymbol{\sigma}|_1\|\cdot\|_{\mathcal{N}_{\Phi_{\boldsymbol{\nu},\boldsymbol{\lambda}}}(\Gamma^d)}\). The bound is given in terms of the level parameter \(L\), where the relationship between \(L\) and the number of interpolation points \(N\) is given explicitly in Theorem~3, \cite{addy2025lengthscaleinformedsparsegridskernel}. We emphasise further that the restriction to isotropic smoothness \(\nu_j-\alpha_j=c\) for all \(1\leq j\leq d\) is adopted solely for presentational convenience. A generalised bound for arbitrary \(\boldsymbol{\nu},\boldsymbol{\alpha}\in\mathbb{R}_{\geq1/2}^d\) satisfying \(\nu_j\geq\alpha_j\) is given in \cite{addy2026thesis}, supporting the application to settings with anisotropic regularity.

The bound is presented as a sum of error contributions, each corresponding to the upper bound for the error of an ISG defined on a subspace of \(\Gamma^d\) identified by \(\mathfrak{u}\). The cofactors \(2^{-c|\mathbf{p}_{\mathfrak{u}}|}\) control the initial magnitude of these contributions. For non-trivial penalties, the cofactors associated with higher dimensional subspaces \(k\)
become significantly smaller. Consequently, to balance these varying initial contributions across subspaces, LISGs begin allocating points in higher dimensional subspaces only after the smaller subspaces have been resolved to sufficient accuracy. The result is an interpolation scheme designed to optimise for the pre-asymptotic convergence of the interpolant, in contrast to ASGs, which target improved asymptotic behaviour. This distinction is particularly important in very high dimensions, where it is often unrealistic to expect all subspaces to be resolved sufficiently for the asymptotic regime to dominate. Nevertheless, the bound shows that the asymptotic error convergence rate of the method coincides with that of the isotropic case in Theorem~\ref{thm: nobile result}, being governed by the slowest-decaying contribution in \(L\), occurring for \(\mathfrak{u}=\{1,\dots,d\}\). In \cite{addy2025lengthscaleinformedsparsegridskernel}, the method was shown to exhibit dimension-independent \(L_2\)-error convergence in practice on a series of test functions with anisotropy encoded through linearly and exponentially growing lengthscales in \(j\), in dimensions up to $d=100$.

\section{Doubly anisotropic sparse grids}\label{sec: dasg construciton}

\begin{figure}
    \centering
    \begin{tikzpicture}[scale=0.8]


    \draw[gray, thick] (-5,5.5+-2) -- (-5,5.5+2);
    \draw[gray, thick] (-1,5.5+-2) -- (-1,5.5+2);
    \draw[gray, thick] (-5,5.5+2) -- (-1,5.5+2);
    \draw[gray, thick] (-5,5.5+-2) -- (-1,5.5+-2);
    \filldraw[gray, fill opacity = 0.2] (-5,5.5+2) -- (-1,5.5+2) -- (-1,5.5+-2) -- (-5,5.5+-2) -- (-5,5.5+2);
    \filldraw[white] (-4,5.5+0.5) -- (-2,5.5+0.5) -- (-2,5.5+-0.5) -- (-4,5.5+-0.5) -- (-4,5.5+0.5);

    \draw[gray, thick] (5,5.5+-2) -- (5,5.5+2);
    \draw[gray, thick] (1,5.5+-2) -- (1,5.5+2);
    \draw[gray, thick] (5,5.5+2) -- (1,5.5+2);
    \draw[gray, thick] (5,5.5+-2) -- (1,5.5+-2);

    \draw[gray, thick] (-4,5.5+-0.5) -- (-4,5.5+0.5);
    \draw[gray, thick] (-2,5.5+-0.5) -- (-2,5.5+0.5);
    \draw[gray, thick] (-4,5.5+0.5) -- (-2,5.5+0.5);
    \draw[gray, thick] (-4,5.5+-0.5) -- (-2,5.5+-0.5);

    \draw[gray, thick] (-4,5.5+-0.5) -- (1,5.5+-2);
    \draw[gray, thick] (-2,5.5+-0.5) -- (1,5.5+-0.5-0.64285714285);
    \draw[gray, thick, dotted] (1,5.5+-0.5-0.64285714285) -- (5,5.5+-2);
    \draw[gray, thick] (-4,5.5+0.5) -- (1,5.5+2);
    \draw[gray, thick] (-2,5.5+0.5) -- (1,5.5+0.5+0.64285714285);
    \draw[gray, thick,dotted] (1,5.5+0.5+0.64285714285) -- (5,5.5+2);

    
    \filldraw[black] (-3,5.5+0) circle (3/4pt);
    
    \filldraw[gray] (-2,5.5+0) circle (3/4pt);
    \filldraw[gray] (-4,5.5+0) circle (3/4pt);
    \filldraw[black] (-2.5,5.5+0) circle (3/4pt);
    \filldraw[black] (-3.5,5.5+0) circle (3/4pt);
    \filldraw[gray] (-1.5,5.5+0) circle (3/4pt);
    \filldraw[gray] (-4.5,5.5+0) circle (3/4pt);
    \filldraw[gray] (-3,5.5+1) circle (3/4pt);
    \filldraw[gray] (-3,5.5+-1) circle (3/4pt);
    \filldraw[gray] (-3,5.5+-0.5) circle (3/4pt);
    \filldraw[gray] (-3,5.5+-1.5) circle (3/4pt);
    \filldraw[gray] (-3,5.5+0.5) circle (3/4pt);
    \filldraw[gray] (-3,5.5+1.5) circle (3/4pt);
    \filldraw[black] (-3,5.5+-0.25) circle (3/4pt);
    \filldraw[black] (-3,5.5+0.25) circle (3/4pt);
    \filldraw[gray] (-3,5.5+1.25) circle (3/4pt);
    \filldraw[gray] (-3,5.5+-1.25) circle (3/4pt);
    \filldraw[gray] (-3,5.5+0.75) circle (3/4pt);
    \filldraw[gray] (-3,5.5+-0.75) circle (3/4pt);
    \filldraw[gray] (-3,5.5+-1.75) circle (3/4pt);
    \filldraw[gray] (-3,5.5+1.75) circle (3/4pt);
    \filldraw[black] (-3,5.5+-0.25-0.125) circle (3/4pt);
    \filldraw[black] (-3,5.5+0.25-0.125) circle (3/4pt);
    \filldraw[gray] (-3,5.5+1.25-0.125) circle (3/4pt);
    \filldraw[gray] (-3,5.5+-1.25-0.125) circle (3/4pt);
    \filldraw[gray] (-3,5.5+0.75-0.125) circle (3/4pt);
    \filldraw[gray] (-3,5.5+-0.75-0.125) circle (3/4pt);
    \filldraw[gray] (-3,5.5+-1.75-0.125) circle (3/4pt);
    \filldraw[gray] (-3,5.5+1.75-0.125) circle (3/4pt);
    
    \filldraw[black] (-3,5.5+-0.25+0.125) circle (3/4pt);
    \filldraw[black] (-3,5.5+0.25+0.125) circle (3/4pt);
    \filldraw[gray] (-3,5.5+1.25+0.125) circle (3/4pt);
    \filldraw[gray] (-3,5.5+-1.25+0.125) circle (3/4pt);
    \filldraw[gray] (-3,5.5+0.75+0.125) circle (3/4pt);
    \filldraw[gray] (-3,5.5+-0.75+0.125) circle (3/4pt);
    \filldraw[gray] (-3,5.5+-1.75+0.125) circle (3/4pt);
    \filldraw[gray] (-3,5.5+1.75+0.125) circle (3/4pt);
    \filldraw[black] (-3,5.5+-0.25-0.125-0.0625) circle (3/4pt);
    \filldraw[black] (-3,5.5+-0.25-0.125+0.0625) circle (3/4pt);
    \filldraw[black] (-3,5.5+0.25-0.125-0.0625) circle (3/4pt);
    \filldraw[black] (-3,5.5+0.25-0.125+0.0625) circle (3/4pt);
    \filldraw[gray] (-3,5.5+1.25-0.125-0.0625) circle (3/4pt);
    \filldraw[gray] (-3,5.5+1.25-0.125+0.0625) circle (3/4pt);
    \filldraw[gray] (-3,5.5+-1.25-0.125-0.0625) circle (3/4pt);
    \filldraw[gray] (-3,5.5+-1.25-0.125+0.0625) circle (3/4pt);
    \filldraw[gray] (-3,5.5+0.75-0.125-0.0625) circle (3/4pt);
    \filldraw[gray] (-3,5.5+0.75-0.125+0.0625) circle (3/4pt);
    \filldraw[gray] (-3,5.5+-0.75-0.125-0.0625) circle (3/4pt);
    \filldraw[gray] (-3,5.5+-0.75-0.125+0.0625) circle (3/4pt);
    \filldraw[gray] (-3,5.5+-1.75-0.125-0.0625) circle (3/4pt);
    \filldraw[gray] (-3,5.5+-1.75-0.125+0.0625) circle (3/4pt);
    \filldraw[gray] (-3,5.5+1.75-0.125-0.0625) circle (3/4pt);
    \filldraw[gray] (-3,5.5+1.75-0.125+0.0625) circle (3/4pt);
    
    \filldraw[black] (-3,5.5+-0.25+0.125-0.0625) circle (3/4pt);
    \filldraw[black] (-3,5.5+-0.25+0.125+0.0625) circle (3/4pt);
    \filldraw[black] (-3,5.5+0.25+0.125-0.0625) circle (3/4pt);
    \filldraw[black] (-3,5.5+0.25+0.125+0.0625) circle (3/4pt);
    \filldraw[gray] (-3,5.5+1.25+0.125-0.0625) circle (3/4pt);
    \filldraw[gray] (-3,5.5+1.25+0.125+0.0625) circle (3/4pt);
    \filldraw[gray] (-3,5.5+-1.25+0.125-0.0625) circle (3/4pt);
    \filldraw[gray] (-3,5.5+-1.25+0.125+0.0625) circle (3/4pt);
    \filldraw[gray] (-3,5.5+0.75+0.125-0.0625) circle (3/4pt);
    \filldraw[gray] (-3,5.5+0.75+0.125+0.0625) circle (3/4pt);
    \filldraw[gray] (-3,5.5+-0.75+0.125-0.0625) circle (3/4pt);
    \filldraw[gray] (-3,5.5+-0.75+0.125+0.0625) circle (3/4pt);
    \filldraw[gray] (-3,5.5+-1.75+0.125-0.0625) circle (3/4pt);
    \filldraw[gray] (-3,5.5+-1.75+0.125+0.0625) circle (3/4pt);
    \filldraw[gray] (-3,5.5+1.75+0.125-0.0625) circle (3/4pt);
    \filldraw[gray] (-3,5.5+1.75+0.125+0.0625) circle (3/4pt);

    \filldraw[gray] (-4,5.5+1) circle (3/4pt);
    \filldraw[gray] (-2,5.5+1) circle (3/4pt);
    \filldraw[gray] (-2,5.5+-1) circle (3/4pt);
    \filldraw[gray] (-4,5.5+-1) circle (3/4pt);

    \filldraw[gray] (1/2 + -4,5.5+0 + 1) circle (3/4pt);
    \filldraw[gray] (-1/2 + -4,5.5+0 + 1) circle (3/4pt);

    \filldraw[gray] (0 + -4,5.5+1/2 + 1) circle (3/4pt);
    \filldraw[gray] (0 + -4,5.5+-1/2 + 1) circle (3/4pt);
    \filldraw[gray] (0 + -4,5.5+-0.5/2 + 1) circle (3/4pt);
    \filldraw[gray] (0 + -4,5.5+-1.5/2 + 1) circle (3/4pt);
    \filldraw[gray] (0 + -4,5.5+1.5/2 + 1) circle (3/4pt);
    \filldraw[gray] (0 + -4,5.5+0.5/2 + 1) circle (3/4pt);
    

    \filldraw[gray] (1/2 + -2,5.5+0 + 1) circle (3/4pt);
    \filldraw[gray] (-1/2 + -2,5.5+0 + 1) circle (3/4pt);

    \filldraw[gray] (0 + -2,5.5+1/2 + 1) circle (3/4pt);
    \filldraw[gray] (0 + -2,5.5+-1/2 + 1) circle (3/4pt);
    \filldraw[gray] (0 + -2,5.5+-0.5/2 + 1) circle (3/4pt);
    \filldraw[gray] (0 + -2,5.5+-1.5/2 + 1) circle (3/4pt);
    \filldraw[gray] (0 + -2,5.5+1.5/2 + 1) circle (3/4pt);
    \filldraw[gray] (0 + -2,5.5+0.5/2 + 1) circle (3/4pt);


    \filldraw[gray] (1/2 + -2,5.5+0 - 1) circle (3/4pt);
    \filldraw[gray] (-1/2 + -2,5.5+0 - 1) circle (3/4pt);

    \filldraw[gray] (0 + -2,5.5+1/2 - 1) circle (3/4pt);
    \filldraw[gray] (0 + -2,5.5+-1/2 - 1) circle (3/4pt);
    \filldraw[gray] (0 + -2,5.5+-0.5/2 - 1) circle (3/4pt);
    \filldraw[gray] (0 + -2,5.5+-1.5/2 - 1) circle (3/4pt);
    \filldraw[gray] (0 + -2,5.5+1.5/2 - 1) circle (3/4pt);
    \filldraw[gray] (0 + -2,5.5+0.5/2 - 1) circle (3/4pt);
    

    \filldraw[gray] (1/2 + -4,5.5+0 - 1) circle (3/4pt);
    \filldraw[gray] (-1/2 + -4,5.5+0 - 1) circle (3/4pt);

    \filldraw[gray] (0 + -4,5.5+1/2 - 1) circle (3/4pt);
    \filldraw[gray] (0 + -4,5.5+-1/2 - 1) circle (3/4pt);
    \filldraw[gray] (0 + -4,5.5+-0.5/2 - 1) circle (3/4pt);
    \filldraw[gray] (0 + -4,5.5+-1.5/2 - 1) circle (3/4pt);
    \filldraw[gray] (0 + -4,5.5+1.5/2 - 1) circle (3/4pt);
    \filldraw[gray] (0 + -4,5.5+0.5/2 - 1) circle (3/4pt);
    \filldraw[black] (3,5.5+0) circle (3/4pt);
    
    \filldraw[black] (2,5.5+0) circle (3/4pt);
    \filldraw[black] (4,5.5+0) circle (3/4pt);
    

    \filldraw[black] (3,5.5+1) circle (3/4pt);
    \filldraw[black] (3,5.5+-1) circle (3/4pt);
    \filldraw[black] (3,5.5+-0.5) circle (3/4pt);
    \filldraw[black] (3,5.5+-1.5) circle (3/4pt);
    \filldraw[black] (3,5.5+0.5) circle (3/4pt);
    \filldraw[black] (3,5.5+1.5) circle (3/4pt);
    \filldraw[black] (3,5.5+-0.5+0.5/2) circle (3/4pt);
    \filldraw[black] (3,5.5+-1.5+0.5/2) circle (3/4pt);
    \filldraw[black] (3,5.5+0.5+0.5/2) circle (3/4pt);
    \filldraw[black] (3,5.5+1.5+0.5/2) circle (3/4pt);
    \filldraw[black] (3,5.5+-0.5-0.5/2) circle (3/4pt);
    \filldraw[black] (3,5.5+-1.5-0.5/2) circle (3/4pt);
    \filldraw[black] (3,5.5+0.5-0.5/2) circle (3/4pt);
    \filldraw[black] (3,5.5+1.5-0.5/2) circle (3/4pt);



    \node[black] at (-5,5.5+-2.3) {\small-0.5};
    \node[black] at (-1,5.5+-2.3) {\small0.5};
    \node[black] at (5,5.5+-2.3) {\small0.5};
    \node[black] at (1,5.5+-2.3) {\small-0.5};

    \node[black] at (-5.5,5.5+-2) {\small-0.5};
    \node[black] at (-5.5,5.5+2) {\small0.5};
    \node[black] at (5.5,5.5+-2) {\small-0.5};
    \node[black] at (5.5,5.5+2) {\small0.5};

    \node[black] at (-3,5.5+-2.5) {\Large\(x_1\)};
    \node[black] at (-5.5,5.5+0) {\Large\(x_2\)};
    \node[black] at (3,5.5+-2.5) {\Large\(x_1\)};
    \node[black] at (5.5,5.5+0) {\Large\(x_2\)};


    \node[black] at (-3,8.8) {Anisotropic};
    \node[black] at (-3,8.2) {\(\boldsymbol{\omega}=(2,1),\:\: \mathbf{p}=(0,0)\)};
    \node[black] at (3,8.8) {Doubly anisotropic};
    \node at (3,8.2) {\(\boldsymbol{\omega}=(2,1),\:\: \mathbf{p}=(1,2)\)};


\end{tikzpicture}
\caption{Anisotropic and doubly-anisotropic sparse grids in two dimensions, each of level \(L=5\) and weighted according to \(\boldsymbol{\omega}=(2,1)\). The designs correspond to penalty choices of \(\mathbf{p}=(0,0)\) and \(\mathbf{p}=(1,2)\), respectively.}\label{fig: ASG and DASG}
\end{figure}
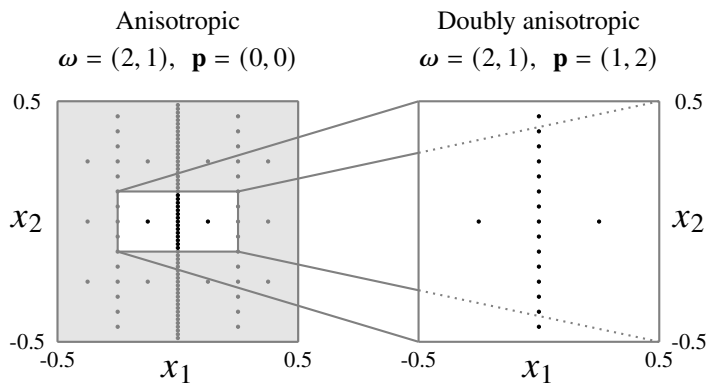

In the previous sections, we introduced Mat\'ern kernel interpolation on ASGs and LISGs, and described how these constructions exploit anisotropy in \(f\), as characterised by the hyperparameters \(\boldsymbol{\nu}\) and \(\boldsymbol{\lambda}\), optimising for faster asymptotic and pre-asymptotic error convergence, respectively. It is therefore natural to consider a generalised sparse grid construction that simultaneously accounts for both types of anisotropy. Under the same lengthscale-anisotropic assumption on \(f\) as in \eqref{eq: lengthscale anisotropy assumption}, \emph{doubly anisotropic sparse grids} (DASGs) achieve this by combining the weighted multi-index set \(\mathcal{A}_{L,\boldsymbol{\omega}}\) defined in \eqref{eq: anisotropic index set} with a non-trivial penalty \(\mathbf{p}\), where the components are again determined by the lengthscales via \(\lambda_j=2^{p_j}\). In each coordinate direction \(j\), the construction is therefore adapted to the relative lengthscale by delaying the onset of the growth points by the penalty \(p_j\), as well as the relative regularity by slowing the rate of growth of points by the weight \(\omega_j\). In Figure~\ref{fig: ASG and DASG}, a two-dimensional DASG  is shown next to an ASG of the same level and weighting.

Similar to Theorem \ref{thm: error in L}, the native space approximation error for DASGs can be bounded component-wise in terms of the approximation error of smaller sparse grids defined over subspaces of \(\Gamma^d\). However, each component now corresponds to an ASG.

\begin{theorem}\label{thm: doubly anisotropic error}
Let \(\mathbf{p}\in\mathbb{N}_0^d\), \(\boldsymbol{\omega}\in\mathbb{R}_{>0}^d\) and \(L\in\mathbb{N}_0\) be given, and let \(\boldsymbol{\alpha},\boldsymbol{\nu}\in\mathbb{R}^d_{\geq1/2}\) be such that \(\alpha_j<\nu_j\) for all \(1\leq j\leq d\). For a constant \(C^{(\boldsymbol{\nu},\boldsymbol{\alpha})}_{\textnormal{L}\ref{lem: tensor hilbert}}\) independent of \(L\) and \(\mathbf{p}\) and defined in Lemma~\ref{lem: tensor hilbert}, we have
\begin{align}
    \left\lVert\, I-S_{\mathcal{A}_{L,\boldsymbol{\omega}},\mathbf{p}}^{(\boldsymbol{\nu},2^\mathbf{p})}\,\right\rVert_{\mathcal{N}_{\Phi_{\boldsymbol{\nu},2^{\mathbf{p}}}}(\Gamma^d)\rightarrow\mathcal{N}_{\Phi_{\boldsymbol{\alpha},2^{\mathbf{p}}}}(\Gamma^d)}\leq C^{(\boldsymbol{\nu},\boldsymbol{\alpha})}_{\textnormal{L}\ref{lem: tensor hilbert}}\sum_{k=1}^d\sum_{\mathfrak{u}\in\mathcal{P}^d_k}2^{-(\boldsymbol{\nu}_{\mathfrak{u}}-\boldsymbol{\alpha}_{\mathfrak{u}})\cdot(\mathbf{p}_{\mathfrak{u}}+\mathbf{1})}\varepsilon^{(k)}_{\boldsymbol{\nu}_{\mathfrak{u}},\boldsymbol{\alpha}_{\mathfrak{u}},\boldsymbol{\omega}_{\mathfrak{u}}}(L-|\mathbf{p}_{\mathfrak{u}}|_1 -k),\nonumber
\end{align}
where \(\varepsilon^{(k)}_{\boldsymbol{\nu}_{\mathfrak{u}},\boldsymbol{\alpha}_{\mathfrak{u}}}(L)\) is 
as defined in \eqref{eq: anisotropic bound}.
\end{theorem}
\begin{proof}
    Given in Section~\ref{sec: double error analysis}
\end{proof}

The benefit of accounting for lengthscale anisotropy remains the same as with LISGs; subspaces of \(\Gamma^d\) associated with a large combined penalty contribute much less to the overall error, and consequently the interpolant is less developed in these subspaces. 
As with LISGs, DASGs inherit the same asymptotic rates as the corresponding ASGs.

In contrast to Theorem~\ref{thm: error in L}, we see an interaction between the two anisotropy types in the factors \(2^{-(\boldsymbol{\nu}_{\mathfrak{u}}-\boldsymbol{\alpha}_{\mathfrak{u}})\cdot(\mathbf{p}_{\mathfrak{u}}+\mathbf{1})}\); the effect of the lengthscale in a given axial direction \(j\) is multiplied by the associated difference in smoothness \(\nu_j-\alpha_j\), compounding the theoretical benefits of large lengthscales when \(\nu_j\) is also large. We note, however, that this construction does not account for how the constant \(C_{\textnormal{L}\ref{lem: supersets double}}^{(\boldsymbol{\nu}_{\mathfrak{u}},\boldsymbol{\alpha}_{\mathfrak{u}})}\) changes with varying regularities. It was found to grow substantially in a brief numerical study in \cite{addy2026thesis}. To account for this growth in practice, we introduce a tuning parameter \(\mathbf{r}\in\mathbb{N}_0^d\) that increases the resolution in each coordinate as appropriate according to \(S_{\mathcal{A}_{L,\boldsymbol{\omega}},\mathbf{p}-\mathbf{r}}^{(\boldsymbol{\nu},2^\mathbf{p})}\). 

The relationship between \(L\) and the number of interpolation points \(N\) is significantly more complicated than in the case of LISGs. For example, if \(\omega_j>1\) for all \(1\leq j\leq d\), then it is now possible that \(S_{\mathcal{A}_{L+1,\boldsymbol{\omega}},\mathbf{p}}^{(\boldsymbol{\nu},2^\mathbf{p})}=S_{\mathcal{A}_{L,\boldsymbol{\omega}},\mathbf{p}}^{(\boldsymbol{\nu},2^\mathbf{p})}\) for a given \(L\in\mathbb{N}_0\). Nonetheless, for the case \(\omega_j\geq 1\) for all \(1\leq j \leq d\), the number of interpolation points in a DASG will always be less than or equal to the number of points in the associated LISG.

\section{Fast implementation}\label{sec:dasg implementation}
As an additional benefit of sparse grid interpolation, the combination of separable kernels and tensor products induces a Kronecker structure in the associated Gram matrices. This structure enables linear solves that are substantially faster than the naive \(\mathcal{O}(N^3)\) cost and, in favourable cases, reduces the computational complexity to nearly \(\mathcal{O}(N\log N)\) \cite{Plumlee2014,addy2025lengthscaleinformedsparsegridskernel}. Evaluating a given interpolant requires computing the weight vector \(\mathbf{w}\in\mathbb{R}^N\) such that,
\[
S_{\mathcal{I},\mathbf{p}}^{(\boldsymbol{\nu},\boldsymbol{\lambda})}(f)(\mathbf{x})=\sum_{i=1}^N w_i\Phi_{\boldsymbol{\nu},\boldsymbol{\lambda}}(\mathbf{x}, \mathbf{x}_i),
\]
for all \(\mathbf{x}\in\Gamma^d\). For the isotropic case, with \(\mathcal{I}=\mathcal{I}_L\) and \(\mathbf{p}=0\), a fast interference algorithm for constructing \(\mathbf{w}\) was developed in \cite{Plumlee2014}. This algorithm was subsequently adapted to LISGs (\(\mathbf{p}\in\mathbb{N}_0^d\)) in \cite{addy2025lengthscaleinformedsparsegridskernel}. In Algorithm~\ref{alg: 1}, we further generalise the approach to anisotropic multi-index sets \(\mathcal{A}_{L,\boldsymbol{\omega}}\), enabling the efficient evaluation of \(\mathbf{w}\) for the DASG interpolant \(S_{\mathcal{A}_{L,\boldsymbol{\omega}},\mathbf{p}}^{(\boldsymbol{\nu},2^{\mathbf{p}})}\). A proof that Algorithm~\ref{alg: 1} produces the correct solution is omitted, as it closely follows the argument presented in \cite{Plumlee2014} and Section 5, \cite{addy2025lengthscaleinformedsparsegridskernel}. The only substantive modification is the appearance of the coefficient \(b_{L,\boldsymbol{\omega},\mathbf{p}}(\boldsymbol{\ell})\), whose form is established in Proposition~\ref{prop: alt DASG def}.

\begin{algorithm}[h]
\caption{Fast inference algorithm for doubly anisotropic sparse grids.}
    \begin{algorithmic}
    \State Initialise $\mathbf{w}=\mathbf{0}\in\mathbb{R}^{N}$
    \For{$\boldsymbol{\ell}\in{\mathcal{W}_{L,\boldsymbol{\omega},\mathbf{p}}}$, defined in Proposition \ref{prop: alt DASG def},}
        \State $\mathbf{w}_{\boldsymbol{\ell}}=\mathbf{w}_{\boldsymbol{\ell}} + b_{L,\boldsymbol{\omega},\mathbf{p}}(\boldsymbol{\ell})\left[\bigotimes_{j=1}^d\phi_{\nu_j,2^{p_j}}(\mathcal{X}_{\ell_j}^{p_j},\mathcal{X}_{\ell_j}^{p_j})^{-1}\right]f(\mathcal{X}_{\ell_1}^{p_j}\times\cdots\times\mathcal{X}_{\ell_d}^{p_j})$
    \EndFor
    \end{algorithmic}
\label{alg: 1}
\end{algorithm}

\begin{proposition}\label{prop: alt DASG def}
Define the set \(\mathcal{W}_{L,\boldsymbol{\omega},\mathbf{p}}=\{\boldsymbol{\ell}\in\mathbb{N}_0^d\,:\,\boldsymbol{\ell}\cdot\boldsymbol{\omega}\leq L, \ell_j> p_j\textrm{ or }\ell_j=0\}\). The sparse grid approximation operator given in Definition~\ref{def: sparse grid operator}, employing the anisotropic multi-index set \(\mathcal{A}_{L,\boldsymbol{\omega}}\), can be rewritten 
    \begin{align}
        S_{\mathcal{A}_{L,\boldsymbol{\omega}},\mathbf{p}}^{(\boldsymbol{\nu},\boldsymbol{\lambda})} &= \sum_{\boldsymbol{\ell}\in \mathcal{W}_{L,\boldsymbol{\omega},\mathbf{p}}}b_{L,\boldsymbol{\omega},\mathbf{p}}(\boldsymbol{\ell})\bigotimes_{j=1}^d s_{\mathcal{X}^{{p_j}}_{\ell_j},\phi_{\nu_j,\lambda_j}},\nonumber
    \end{align}
    where
    \(
        b_{L,\boldsymbol{\omega},\mathbf{p}}(\boldsymbol{\ell})=\sum_{\mathfrak{u
        }\in\mathcal{U}_{L,\boldsymbol{\omega},\mathbf{p}}(\boldsymbol{\ell})}(-1)^{|\mathfrak{u}|},\nonumber
    \)
    and \(\mathcal{U}_{L,\boldsymbol{\omega},\mathbf{p}}(\boldsymbol{\ell})\coloneqq \{\mathfrak{u}\subset\{1,\dots,d\}\,:\,\boldsymbol{\ell}\cdot\boldsymbol{\omega}+\sum_{j\in\mathfrak{u}}[\omega_{j}+\mathbf{1}_{\{\ell_{j}=0\}}p_{j}]\leq L\}\).
\end{proposition}
\begin{proof}
This is an instance of the well-known sparse grid combination technique (see e.g. \cite{pflaum1999,DELVOS198299}), which for completeness we give here. For a given \(\boldsymbol{\ell}\in\mathcal{A}_{L,\boldsymbol{\omega}}\), we wish to find the coefficient of \(s_{\mathcal{X}^{p_1}_{\ell_1},\phi_{\nu_1,\lambda_1}}\otimes\cdots\otimes  s_{\mathcal{X}^{p_d}_{\ell_d},\phi_{\nu_d,\lambda_d}}\), denoted \(b_{L,\boldsymbol{\omega},\mathbf{p}}(\boldsymbol{\ell})\), in the expansion of \(S_{\mathcal{A}_{L,\boldsymbol{\omega}},\mathbf{p}}^{(\boldsymbol{\nu},\boldsymbol{\lambda})}\) given in Definition~\ref{def: sparse grid operator}. 

Let \(\boldsymbol{\ell}\in\mathcal{A}_{L,\boldsymbol{\omega}}\) and define \(\mathfrak{w}\coloneqq\{\mathfrak{w}\subset\{1,\dots,d\}\,:\,0<{\ell}_{i}\leq p_i\textrm{ for all }i\in\mathfrak{w}\}\). Then,
\begin{align}
&\bigotimes_{j=1}^d\left(s_{\mathcal{X}_{\ell_j}^{p_j},\phi_{\nu_j,\lambda_j}}-s_{\mathcal{X}_{\ell_j-1}^{p_j},\phi_{\nu_j,\lambda_j}}\right)\nonumber\\ = &\bigotimes_{j\in\mathfrak{w}}\left(s_{\mathcal{X}_{\ell_j}^{p_j},\phi_{\nu_j,\lambda_j}}-s_{\mathcal{X}_{\ell_j-1}^{p_j},\phi_{\nu_j,\lambda_j}}\right) \otimes \bigotimes_{j\in\{1,\dots,d\}\setminus\mathfrak{w}}\left(s_{\mathcal{X}_{\ell_j}^{p_j},\phi_{\nu_j,\lambda_j}}-s_{\mathcal{X}_{\ell_j-1}^{p_j},\phi_{\nu_j,\lambda_j}}\right).\nonumber
\end{align}
By definition of the point sets, for all \(j\in\mathfrak{w}\), we have \(\mathcal{X}_{\ell_j}^{p_j}=\mathcal{X}_{\ell_j+1}^{p_j}=\{0\}\), and therefore \(s_{\mathcal{X}_{\ell_j}^{p_j},\phi_{\nu_j,\lambda_j}}-s_{\mathcal{X}_{\ell_j-1}^{p_j},\phi_{\nu_j,\lambda_j}} = 0\). Thus the only non-zero contributions correspond to multi-indices in the subset \(\mathcal{W}_{L,\boldsymbol{\omega},\mathbf{p}}\subset\mathcal{A}_{L,\boldsymbol{\omega}}\).

Now, let \(\ell_j> p_j\) for all \(1\leq j\leq d\). By construction of the operators as differences between successive interpolants, all non-zero contributions to \(b_{L,\boldsymbol{\omega},\mathbf{p}}(\boldsymbol{\ell})\) arise from the multi-indices \(\mathbf{a}\in\{\mathbf{a}\in\mathcal{W}_{L,\boldsymbol{\omega},\mathbf{p}}\,:\,|\mathbf{a}-\boldsymbol{\ell}|_{\ell^\infty}\in\{0,1\}\}\). For each such multi-index \(\mathbf{a}\), 
\begin{align}
\bigotimes_{j=1}^d\left(s_{\mathcal{X}_{a_j}^{p_j},\phi_{\nu_j,\lambda_j}}-s_{\mathcal{X}_{a_j-1}^{p_j},\phi_{\nu_j,\lambda_j}}\right)=\sum_{\tilde{\mathbf{a}}\in\{0,1\}^d}(-1)^{|\tilde{\mathbf{a}}|_1}\bigotimes_{i=1}^ds_{\mathcal{X}_{a_j-\tilde{a}_j}^{p_j},\phi_{\nu_j,\lambda_j}},\label{eq: binary expansion}
\end{align}
and hence the contribution to \(b_{L,\boldsymbol{\omega},\mathbf{p}}(\boldsymbol{\ell})\) associated to the multi-index \(\mathbf{a}\) is the cofactor in the sum in \eqref{eq: binary expansion} for \(\tilde{\mathbf{a}}=\mathbf{a}-\boldsymbol{\ell}\), given by \((-1)^{|\mathbf{a}-\boldsymbol{\ell}|_1}\). By noticing that each multi-index \(\mathbf{a}\) is uniquely determined by a subset \(\mathfrak{u}\subset\mathcal{U}_{L,\boldsymbol{\omega},\mathbf{p}}(\boldsymbol{\ell})\), where \(j\in\mathfrak{u}\) indicates that \(a_j-l_j=1\), and \(j\in\{0,\dots,d\}\setminus\mathfrak{u}\) indicates \(a_j-l_j=0\), we arrive at the statement of the proposition.

Finally, the same logic applies for \(\boldsymbol{\ell}\in\mathcal{W}_{L,\boldsymbol{\omega},\mathbf{p}}\) such that \(\ell_{\mathfrak{v}}=\mathbf{0}\in\mathbb{N}_0^k\) for some subset \(\mathfrak{v}\subset\{1,\dots,d\}\) of size \(k\), where the telescoping of sums means that the non-zero contributixons arise more generally from the multi-indices \(\mathbf{a}\in\{\mathbf{a}\in\mathcal{W}_{L,\boldsymbol{\omega},\mathbf{p}}\,:|a_j-\ell_j|\in\{0,p_j\}\textrm{ for all }j\in\mathfrak{v}\textrm{ and }\,|{a}_{j}-{\ell}_{j}|\in\{0,1\}\textrm{ for all }j\in\{1,\dots,d\}\setminus\mathfrak{v}\}\).
\end{proof}

\section{Error analysis}\label{sec: double error analysis}

Here we give the proof of Theorem~\ref{thm: doubly anisotropic error}, 
adapted from that of Theorem~\ref{thm: error in L} in \cite{addy2025lengthscaleinformedsparsegridskernel}.

\begin{lemma}[Adapted from Lemma 4, \cite{addy2025lengthscaleinformedsparsegridskernel}]\label{lem: tensor hilbert}
Let \(L\in\mathbb{N}\), \(\mathbf{p}\in\mathbb{N}_0^d\), \(\Gamma\coloneqq (-1/2,1/2)\) and let \(\boldsymbol{\nu},\boldsymbol{\alpha}\in\mathbb{R}^d_{\geq1/2}\) be such that \(\alpha_j\leq\nu_j\) for all \(1\leq j \leq d\). Define the multi-index set \(\mathcal{K}^d_k\coloneqq\{\boldsymbol{l}\in\mathbb{N}_{0}^d:|\{1\leq j\leq d:l_j\neq0\}|=k\}\). Then,
    \begin{align}
    \left\lVert\, I-S_{\mathcal{A}_{L,\boldsymbol{\omega}},\mathbf{p}}^{(\boldsymbol{\nu},2^\mathbf{p})}\,\right\rVert_{\mathcal{N}_{\Phi_{\boldsymbol{\nu},2^{\mathbf{p}}}}(\Gamma^d)\rightarrow\mathcal{N}_{\Phi_{\boldsymbol{\alpha},2^{\mathbf{p}}}}(\Gamma^d)}\leq C^{(\boldsymbol{\nu},\boldsymbol{\alpha})}_{\textnormal{L}\ref{lem: tensor hilbert}}\sum_{k=1}^d\sum_{\boldsymbol{l}\in\mathcal{K}_k^d\setminus \mathcal{A}_{L,\boldsymbol{\omega}}}\prod_{j=1}^d\mathbf{Q}_{\boldsymbol{\nu},\boldsymbol{\alpha}}(\boldsymbol{l},\mathbf{p})_j,\nonumber
    \end{align}
    where
    \begin{align}
        \mathbf{Q}_{\boldsymbol{\nu},\boldsymbol{\alpha}}(\boldsymbol{l},\mathbf{p})_j \coloneqq \begin{cases}
            C^{(\nu_j,\alpha_j)}_{\textnormal{W}}2^{-(\nu_j-\alpha_j)l_j} & \textrm{ if }l_j\geq p_j+1,\\
            0 & \textrm{ if }1\leq l_j\leq p_j, \textrm{ and}\\
            1 & \textrm{ if }l_j = 0,
        \end{cases}\nonumber
    \end{align}
    and the constant \(C^{(\nu,\alpha)}_{\textnormal{W}}\) is independent of \(\mathbf{p}\) (equal to \(C^{(\nu,\alpha)}_{3}\) in \cite{addy2025lengthscaleinformedsparsegridskernel}), and the constant \(C^{(\boldsymbol{\nu},\boldsymbol{\alpha})}_{\textnormal{L}\ref{lem: tensor hilbert}}\) is given by
    \(
        C^{(\boldsymbol{\nu},\boldsymbol{\alpha})}_{\textnormal{L}\ref{lem: tensor hilbert}} = \prod_{j=1}^d\sqrt{\frac{\Gamma(\alpha_j+1/2)\Gamma(\nu_j)}{\Gamma(\alpha_j)\Gamma(\nu_j+1/2)}}.\nonumber
    \)
\end{lemma}
\begin{proof}
    This follows that of Lemma~4, replacing the multi-index set $\mathcal{I}^d_L$ with $\mathcal{A}_{L,\boldsymbol{\omega}}$.
\end{proof}

\begin{lemma}[Lemma 7, \cite{addy2025lengthscaleinformedsparsegridskernel}]\label{lem: bijection}
    Let \(d\in\mathbb{N}\) and \(k\in\mathbb{N}_0\). We have that \(\mathcal{K}_k^d\cong\{\mathbf{0}\}\cup\mathcal{P}^d_k\times\mathbb{N}^k\) with bijection, for a given \(\boldsymbol{l}\in\mathbb{N}_0^d\),
    \begin{align}
        \boldsymbol{\rho}(\boldsymbol{l}) = \begin{cases}
            \mathbf{0}\in\mathbb{N}_0^k&\textrm{if }\boldsymbol{l}=\mathbf{0}, \textrm{ and}\\
            \left(\mathfrak{u}(\boldsymbol{l}),\boldsymbol{l}_{\mathfrak{u}(\boldsymbol{l})}\right)&\textrm{otherwise},
            \end{cases}\nonumber
    \end{align}
    where \(\mathfrak{u}(\boldsymbol{l})\coloneqq\{j\in\{1,\dots,d\}:l_j=0\}\subset\mathcal{P}^d_k\) and, if \(|\mathfrak{u}(\boldsymbol{l})|=k\), we define, for a given \(\mathbf{b}\in\mathbb{N}_0^d\), \(\mathbf{b}_{\mathfrak{u}(\boldsymbol{l})}\coloneqq(b_{\mathfrak{u}(\boldsymbol{l})_1},\dots,b_{\mathfrak{u}(\boldsymbol{l})_k})\).
\end{lemma}

\begin{corollary}[Adapted from Corollary 4, \cite{addy2025lengthscaleinformedsparsegridskernel}]\label{cor: level sets double}
    Let \(d,k\in\mathbb{N}\), \(L\in\mathbb{N}_0\), and fix \(\boldsymbol{\omega}\in\mathbb{R}_{>0}^d\). Let \(\mathbf{0}\neq\boldsymbol{l}\in\mathcal{K}_k^d\) and define \(\boldsymbol{\rho}(\boldsymbol{l})=(\mathfrak{u},\mathbf{a})\). Then \(\boldsymbol{l}\in\mathcal{A}_{L,\boldsymbol{\omega}}\) if and only if \(\mathbf{a}\in\mathcal{A}_{L,\boldsymbol{\omega}_{\mathfrak{u}}}\subset\mathbb{N}_0^k\).
\end{corollary}

\begin{proof}
Let \(\boldsymbol{l}\in \mathcal{A}_{L,\boldsymbol{\omega}}\). Then \(\sum_{j=1}^dl_j\omega_j\leq L\).
By definition of \(\boldsymbol{\rho}\), \(\mathbf{a}=\boldsymbol{l}_{\mathfrak{u}(\boldsymbol{l})}\), and hence \(\sum_{i=1}^ka_i\omega_{\mathfrak{u}_i}=\sum_{i=1}^kl_{\mathfrak{u}(\boldsymbol{l})_i}\omega_{\mathfrak{u}_i}\leq\sum_{j=1}^dl_j\omega_j\leq L\), since \(l_j\geq0\) and \(\omega_j>0\). Thus \(\mathbf{a}\in\mathcal{A}_{L,\boldsymbol{\omega}_{\mathfrak{u}}}\).

For the opposite direction, let \(\mathbf{a}\in\mathcal{A}_{L,\boldsymbol{\omega}_{\mathfrak{u}}}\). By definition of \(\boldsymbol{\rho}\), \(l_j=0\) if \(j\notin\mathfrak{u}(\boldsymbol{l})\). Since \(\boldsymbol{l}\in\mathcal{K}^d_k\), we have \(|\mathfrak{u}(\boldsymbol{l})|=k\), and hence \(\sum_{j=1}^dl_j\omega_j=\sum_{i=1}^kl_{\mathfrak{u}(\boldsymbol{l})_i}\omega_{\mathfrak{u}_i}=\sum_{i=1}^ka_i\omega_{\mathfrak{u}_i}\leq L\). Thus \(\boldsymbol{l}\in\mathcal{A}_{L,\boldsymbol{\omega}}\) also.
\end{proof}

\begin{corollary}[Adapted from Corollary 5, \cite{addy2025lengthscaleinformedsparsegridskernel}]\label{cor: isomporphic double}
    Let \(d,k\in\mathbb{N}\), \(L\in\mathbb{N}_0\), and fix \(\boldsymbol{\omega}\in\mathbb{R}_{>0}^d\). Then
    \begin{align}
    \mathcal{K}_k^d\cap\mathcal{A}_{L,\boldsymbol{\omega}}\cong\{\mathbf{0}\}\cup\bigsqcup_{\mathfrak{u}\in\mathcal{P}^d_k}\mathbb{N}^k\cap\mathcal{A}_{L,\boldsymbol{\omega}_{\mathfrak{u}}}\nonumber,
    \end{align}
    and
    \begin{align}
    \mathcal{K}_k^d\setminus\mathcal{A}_{L,\boldsymbol{\omega}}\cong\bigsqcup_{\mathfrak{u}\in\mathcal{P}^d_k}\mathbb{N}^k\setminus\mathcal{A}_{L,\boldsymbol{\omega}_{\mathfrak{u}}},\nonumber
    \end{align}
    where the disjoint unions are defined by \(\bigsqcup_{a\in A}B_a\coloneqq\bigcup_{a\in A}\{(a,b)\,:\,b\in B_a\}\).
\end{corollary}

\begin{proof}
    This follows immediately from Lemma~\ref{lem: bijection} and Corollary~\ref{cor: level sets double} with \(\boldsymbol{\rho}\) acting as a bijection in both cases, where we note that \(\mathbf{0}\in\mathcal{A}_{L,\boldsymbol{\omega}_{\mathfrak{u}}}\) for all \(L>0\).
\end{proof}

\begin{lemma}[Adapted from Lemma 5, \cite{addy2025lengthscaleinformedsparsegridskernel}]\label{lem: supersets double}
    Let \(L\in\mathbb{N}_{0}\), \(\mathbf{p}\in\mathbb{N}_0^d\), and \(\boldsymbol{\omega}\in\mathbb{R}_{>0}^d\). Then,
    \begin{align}
        \sum_{\boldsymbol{l}\in\mathcal{K}^d_k\setminus \mathcal{A}_{L,\boldsymbol{\omega}}}\prod_{j=1}^d{\mathbf{Q}_{\boldsymbol{\nu},\boldsymbol{\alpha}}(\boldsymbol{l},\mathbf{p})_j} =\sum_{\mathfrak{u}\in\mathcal{P}^d_k}C_{\textnormal{L}\ref{lem: supersets double}}^{(\boldsymbol{\nu}_{\mathfrak{u}},\boldsymbol{\alpha}_{\mathfrak{u}})}\sum_{\boldsymbol{l}\in\mathbb{N}_0^k\setminus\mathcal{A}_{\max\{0,L-|\mathbf{p}_{\mathfrak{u}}|-k\},\boldsymbol{\omega}_{\mathfrak{u}}}}2^{-(\boldsymbol{\nu}_{\mathfrak{u}}-\boldsymbol{\alpha}_{\mathfrak{u}})\cdot(\boldsymbol{l}+\mathbf{p}_{\mathfrak{u}}+\mathbf{1})},\nonumber
    \end{align}
    where \(C_{\textnormal{L}\ref{lem: supersets double}}^{(\boldsymbol{\nu}_{\mathfrak{u}},\boldsymbol{\alpha}_{\mathfrak{u}})}\coloneqq\prod_{i=1}^kC_{\textnormal{W}}^{(\nu_{\mathfrak{u}_i},\alpha_{\mathfrak{u}_i})}\).
\end{lemma}

\begin{proof}
    This proof is almost identical to that of Lemma 5, \cite{addy2025lengthscaleinformedsparsegridskernel}, where the multi-index set \(\mathcal{I}_L^d\) is replaced by \(\mathcal{A}_{L,\boldsymbol{\omega}}\). The complete proof is given in \cite{addy2026thesis}.
\end{proof}

\begin{proof}[Proof of Theorem~\ref{thm: doubly anisotropic error}]
    Simply combining Lemmas~\ref{lem: tensor hilbert} and \ref{lem: supersets double} gives

    \begin{align}
    \begin{split}
    &\quad\left\lVert\, I-S_{\mathcal{A}_{L,\boldsymbol{\omega}},\mathbf{p}}^{(\boldsymbol{\nu},2^\mathbf{p})}\,\right\rVert_{\mathcal{N}_{\Phi_{\boldsymbol{\nu},2^{\mathbf{p}}}}(\Gamma^d)\rightarrow\mathcal{N}_{\Phi_{\boldsymbol{\alpha},2^{\mathbf{p}}}}(\Gamma^d)}\\
    &\leq C^{(\boldsymbol{\nu},\boldsymbol{\alpha})}_{\textnormal{L}\ref{lem: tensor hilbert}}\sum_{k=1}^d\sum_{\boldsymbol{l}\in\mathcal{K}_k^d\setminus \mathcal{A}_{L,\boldsymbol{\omega}_{\mathfrak{u}}}}\prod_{j=1}^d\mathbf{Q}_{\boldsymbol{\nu},\boldsymbol{\alpha}}(\boldsymbol{l},\mathbf{p})_j,
    \end{split}\nonumber\\
    &\leq C^{(\boldsymbol{\nu},\boldsymbol{\alpha})}_{\textnormal{L}\ref{lem: tensor hilbert}}\sum_{k=1}^d\sum_{\mathfrak{u}\in\mathcal{P}^d_k}C_{\textnormal{L}\ref{lem: supersets double}}^{(\boldsymbol{\nu}_{\mathfrak{u}},\boldsymbol{\alpha}_{\mathfrak{u}})}\sum_{\boldsymbol{l}\in\mathbb{N}_0^k\setminus\mathcal{A}_{\max\{0,L-|\mathbf{p}_{\mathfrak{u}}|-k\},\boldsymbol{\omega}_{\mathfrak{u}}}}2^{-(\boldsymbol{\nu}_{\mathfrak{u}}-\boldsymbol{\alpha}_{\mathfrak{u}})\cdot(\boldsymbol{l}+\mathbf{p}_{\mathfrak{u}}+\mathbf{1})},\nonumber\\
    &=C^{(\boldsymbol{\nu},\boldsymbol{\alpha})}_{\textnormal{L}\ref{lem: tensor hilbert}}\sum_{k=1}^d\sum_{\mathfrak{u}\in\mathcal{P}^d_k}2^{-(\boldsymbol{\nu}_{\mathfrak{u}}-\boldsymbol{\alpha}_{\mathfrak{u}})\cdot(\mathbf{p}_{\mathfrak{u}}+\mathbf{1})}C_{\textnormal{L}\ref{lem: supersets double}}^{(\boldsymbol{\nu}_{\mathfrak{u}},\boldsymbol{\alpha}_{\mathfrak{u}})}\sum_{\boldsymbol{l}\in\mathbb{N}_0^k\setminus\mathcal{A}_{\max\{0,L-|\mathbf{p}_{\mathfrak{u}}|-k\},\boldsymbol{\omega}_{\mathfrak{u}}}}2^{-(\boldsymbol{\nu}_\mathfrak{u}-\boldsymbol{\alpha}_{\mathfrak{u}})\cdot\boldsymbol{l}}.\nonumber
    \end{align}
    Substituting in the anisotropic sparse grid error bound \(\varepsilon^{(k)}_{\boldsymbol{\nu}_{\mathfrak{u}},\boldsymbol{\alpha}_{\mathfrak{u}},\boldsymbol{\omega}_{\mathfrak{u}}}(L-|\mathbf{p}_{\mathfrak{u}}|_1 -k)\) in \eqref{eq: anisotropic bound} gives us the statement of the theorem.
\end{proof}

\section{Numerical experiments}\label{sec: double experiments}


In Figure~\ref{fig: DASG vs LISG}, we numerically assess the performance of interpolants constructed on DASG designs against those defined on ISG, ASG and LISG designs, respectively. For target functions, we use realisations of random linear combinations of \(d\)-dimensional separable Mat\'ern kernels, defined by
\begin{align}
f_{\Phi,\boldsymbol{\nu},\mathbf{p}}\coloneqq\sum_{i=1}^{20}\xi_i\Phi_{\boldsymbol{\nu},2^{\mathbf{p}}}(\cdot,\mathbf{y}_i),\label{eq: d dim kernel function}
\end{align}
where \(\xi_{i}\sim\mathcal{N}(0,5)\) and \(\mathbf{y}_{i}\sim U(\Gamma^d)\) are i.i.d. for \(1\leq i\leq 20\). These are the same test functions used in \cite{addy2025lengthscaleinformedsparsegridskernel}, and are particularly useful in this setting since the anisotropic structure characterised by \(\boldsymbol{\nu}\) and \(\boldsymbol{\lambda}\) is known exactly \emph{a priori}. The relative \(L^2\)-errors presented are averaged over 10 realisations \(f\) of \(f_{\Phi,\boldsymbol{\nu},\mathbf{p}}\), and each error is approximated by 100 Monte Carlo samples. In all experiments, we consider only lengthscale-anisotropy in \(f\) specified by linearly growing penalties; \(\mathbf{p}=\mathbf{p}_{\textrm{lin}}=(0,1,\dots,d-1)\), corresponding to exponentially growing lengthscales in the dimension coordinate, \(j\). Two parametrisations for the smoothness anisotropy are considered: \(\boldsymbol{\nu}=(\mathbf{3}/2,\mathbf{5}/2)\in\mathbb{R}^{d}\) and  \(\boldsymbol{\nu}=(\mathbf{5}/2,\mathbf{3}/2)\in\mathbb{R}^{d}\), where \(\mathbf{5}/2,\mathbf{3}/2\in\mathbb{R}^{d/2}\). These correspond to two distinct regimes. In the first, both the regularity and lengthscales grow with \(j\), and in the second, the regularity decreases as the lengthscale increases. {We emphasise that, relative to the degree of anisotropy in the lengthscales, the anisotropy in \(\boldsymbol{\nu}\) mild; in particular, the smoothness in a given coordinate direction \(j\) never exceeds \(\nu_j=5/2\). 
The primary objective of the experiments is to investigate the extent to which the approximation error for ASG and LISG methods may be reduced by combining their respective methodologies. The focus is therefore on this interaction, rather than the potential dimension insensitivity that may arise from having the regularity increase in \(j\), as is the case in \cite{Nobile2008}, or dimension-independence as a consequence of lengthscale-anisotropy, which has been studied in \cite{addy2025lengthscaleinformedsparsegridskernel}}. To counteract the growth of the constants \(C_{\textnormal{L}\ref{lem: supersets double}}^{(\boldsymbol{\nu}_{\mathfrak{u}},\boldsymbol{\alpha}_{\mathfrak{u}})}\), tuning parameters \(\mathbf{r}=(\mathbf{0},\mathbf{2})\) with \(\mathbf{0},\mathbf{2}\in\mathbb{N}_0^{d/2}\) were used only in the case \(\boldsymbol{\nu}=(\mathbf{3}/2,\mathbf{5}/2)\).  Each kernel interpolant was evaluated using Algorithm~\ref{alg: 1} (or the corresponding version thereof) and constructed for increasing levels until either \(N\) exceeded \(10^5\) or {the corresponding Gram matrix ceased to be positive definite in double precision}. 

The most striking effect observed in Figure \ref{fig: DASG vs LISG} is the significantly lower error achieved by the methods employing kernels with anisotropic lengthscales (DASG and LISG) compared with those that do not (ASG and LISG). 
Notably, DASGs appear to inherit the favourable pre-asymptotic behaviour of LISGs.

In our numerical experiments, substantial differences were observed in the efficacy of DASGs depending on the orientation of the anisotropies. In particular, when both regularities and lengthscales grow with \(j\), highly inconsistent error convergence rates in \(N\) were observed already in \cite{addy2026thesis}, and was attributed to unaccounted-for growth in the constants \(C_{\textnormal{L}\ref{lem: supersets double}}^{(\boldsymbol{\nu}_{\mathfrak{u}},\boldsymbol{\alpha}_{\mathfrak{u}})}\). Supporting this interpretation, slightly reducing the penalties in the smoother directions in Figure \ref{fig: DASG vs LISG} via the tuning parameter \(\mathbf{r}\), as described above, substantially mitigates this effect. This selection of \(\mathbf{r}\) was informed by exploratory, non-exhaustive experimentation.

In the case \(d=4\), DASGs achieve an improved convergence rate compared to LISGs for both anisotropy configurations considered. For \(d=8\) and \(d=16\), however, the observed differences in the error between the two methods are minimal. This supports our emphasis on the importance of pre-asymptotic convergence in high dimensions, as discussed in Subsection~\ref{subsec: LISGs}, where it was suggested that the onset of the asymptotic regime becomes increasingly delayed as \(d\) increases. A clear advantage of DASGs over LISGs is that, by placing relatively fewer points in the smoothest directions, DASGs appear significantly more resilient to ill-conditioning of the associated Gram matrices. In practice, this often enables the construction of interpolants with a substantially larger number of interpolation points \(N\). 

Somewhat unexpectedly, ISGs consistently yield a smaller observed error than ASGs in this setting, most apparent when \(d=4\). This suggests that the onset of the asymptotic regime may also be delayed when the lengthscales \(\boldsymbol{\lambda}\), determined by the penalty \(\mathbf{p}\), are poorly chosen---here, taken to be significantly smaller than those of \(f_{\Phi,\boldsymbol{\nu},\mathbf{p}}\).

\begin{figure}[H]
    \begin{subfigure}{.45\textwidth}
    \begin{tikzpicture}[scale=0.75]
        \begin{axis}[
        cycle list name = myCycleList,
        grid = both,
        xmode = log, ymode = log,
        xmin = 0.8, xmax = 1e5,
        ymin = 1e+-10, ymax= 1,
        legend pos=south west,
        xticklabels={}
        ]
        \addplot table [x index=1,y index=0, col sep=space] {author/L2_error_dim_4_lin_p_same_orient_plus_2.txt};
        \addplot[
            only marks,
            mark=x,
            mark options={scale=3, line width = 2pt},
            color = darkgray,
            forget plot
        ] table [
            x index=1,
            y index=0,
            col sep=space,
            skip first n=19
        ] {author/L2_error_dim_4_lin_p_same_orient_plus_2.txt};
        \addplot table [x index=1,y index=0, col sep=space] {author/L2_error_dim_4_lin_p_same_orient_LISG.txt};
        \addplot[
            only marks,
            mark=x,
            mark options={scale=3, line width = 2pt},
            color = skyblue,
            forget plot
        ] table [
            x index=1,
            y index=0,
            col sep=space,
            skip first n=8
        ] {author/L2_error_dim_4_lin_p_same_orient_LISG.txt};
        \addplot table [x index=1,y index=0, col sep=space] {author/L2_error_dim_4_lin_p_same_orient_ASG.txt};
        \addplot table [x index=1,y index=0, col sep=space] {author/L2_error_dim_4_lin_p_same_orient_ISG.txt};

        \legend{DASG, LISG, ASG, ISG};
        \end{axis}
        \node at (3.5,6.5) {\Large\(\boldsymbol{\nu}=(\mathbf{3}/2,\mathbf{5}/2)\)};
        \node[rotate=90] at (-1.3,2.8) {\large Approx. \(L^2\)-Error};
    \end{tikzpicture}
    \end{subfigure}
    \hspace{10mm}
    \begin{subfigure}{.45\textwidth}
    \begin{tikzpicture}[scale=0.75]
        \begin{axis}[
        cycle list name = myCycleList,
        grid = both,
        xmode = log, ymode = log,
        xmin = 0.8, xmax = 1e5,
        ymin = 1e+-10, ymax= 1,
        legend pos=south west,
        xticklabels={}, yticklabels={}
        ]
        \addplot table [x index=1,y index=0, col sep=space] {author/L2_error_dim_4_lin_p_opposite_orient_plus_0.txt};
        \addplot table [x index=1,y index=0, col sep=space] {author/L2_error_dim_4_lin_p_same_orient_LISG.txt};
        \addplot[
            only marks,
            mark=x,
            mark options={scale=3, line width = 2pt},
            color = skyblue,
            forget plot
        ] table [
            x index=1,
            y index=0,
            col sep=space,
            skip first n=8
        ] {author/L2_error_dim_4_lin_p_opposite_orient_LISG.txt};
        \addplot table [x index=1,y index=0, col sep=space] {author/L2_error_dim_4_lin_p_opposite_orient_ASG.txt};
        \addplot table [x index=1,y index=0, col sep=space] {author/L2_error_dim_4_lin_p_same_orient_ISG.txt};
        
        \legend{DASG, LISG, ASG, ISG};
        \end{axis}
        \node at (3.5,6.5) {\Large$\boldsymbol{\nu}=(\mathbf{5}/2,\mathbf{3}/2)$};
        \node[rotate=-90] at (7.7,2.8) {\Large$d=4$};
    \end{tikzpicture}
    \end{subfigure}\\
    \begin{subfigure}{.45\textwidth}
    \begin{tikzpicture}[scale=0.75]
        \begin{axis}[
        cycle list name = myCycleList,
        grid = both,
        xmode = log, ymode = log,
        xmin = 0.8, xmax = 1e5,
        ymin = 1e+-10, ymax= 1,
        legend pos=south west,
        xticklabels = {}
        ]
        \addplot table [x index=1,y index=0, col sep=space] {author/L2_error_dim_8_lin_p_same_orient_plus_2.txt};
        \addplot table [x index=1,y index=0, col sep=space] {author/L2_error_dim_8_lin_p_same_orient_LISG.txt};
        \addplot[
            only marks,
            mark=x,
            mark options={scale=3, line width = 2pt},
            color = skyblue,
            forget plot
        ] table [
            x index=1,
            y index=0,
            col sep=space,
            skip first n=8
        ] {author/L2_error_dim_8_lin_p_same_orient_LISG.txt};
        \addplot table [x index=1,y index=0, col sep=space] {author/L2_error_dim_8_lin_p_same_orient_ASG.txt};
        \addplot table [x index=1,y index=0, col sep=space] {author/L2_error_dim_8_lin_p_same_orient_ISG.txt};

        \legend{DASG,LISG,ASG,ISG};
        \end{axis}
        \node[rotate=90] at (-1.3,2.8) {\large Approx. \(L^2\)-Error};
    \end{tikzpicture}
    \end{subfigure}
    \hspace{10mm}
    \begin{subfigure}{.45\textwidth}
    \begin{tikzpicture}[scale=0.75]
        \begin{axis}[
        cycle list name = myCycleList,
        grid = both,
        xmode = log, ymode = log,
        xmin = 0.8, xmax = 1e5,
        ymin = 1e+-10, ymax= 1,
        legend pos=south west,
        yticklabels = {}, xticklabels = {}
        ]
        \addplot table [x index=1,y index=0, col sep=space] {author/L2_error_dim_8_lin_p_opposite_orient_plus_0.txt};
        \addplot[
            only marks,
            mark=x,
            mark options={scale=3, line width = 2pt},
            color = darkgray,
            forget plot
        ] table [
            x index=1,
            y index=0,
            col sep=space,
            skip first n=24
        ] {author/L2_error_dim_8_lin_p_opposite_orient_plus_0.txt};
        \addplot table [x index=1,y index=0, col sep=space] {author/L2_error_dim_8_lin_p_opposite_orient_LISG.txt};
        \addplot[
            only marks,
            mark=x,
            mark options={scale=3, line width = 2pt},
            color = skyblue,
            forget plot
        ] table [
            x index=1,
            y index=0,
            col sep=space,
            skip first n=8
        ] {author/L2_error_dim_8_lin_p_opposite_orient_LISG.txt};
        \addplot table [x index=1,y index=0, col sep=space] {author/L2_error_dim_8_lin_p_opposite_orient_ASG.txt};
        \addplot table [x index=1,y index=0, col sep=space] {author/L2_error_dim_8_lin_p_same_orient_ISG.txt};
        
        \legend{DASG,LISG,ASG,ISG};
        \end{axis}
        \node[rotate=-90] at (7.7,2.8) {\Large$d=8$};
    \end{tikzpicture}
    \end{subfigure}\\
    \begin{subfigure}{.45\textwidth}
    \begin{tikzpicture}[scale=0.75]
        \begin{axis}[
        cycle list name = myCycleList,
        grid = both,
        xmode = log, ymode = log,
        xmin = 0.8, xmax = 1e5,
        ymin = 1e+-10, ymax= 1,
        xlabel={\Large$N$},
        legend pos=south west
        ]
        \addplot table [x index=1,y index=0, col sep=space] {author/L2_error_dim_16_lin_p_same_orient_plus_2.txt};
        \addplot table [x index=1,y index=0, col sep=space] {author/L2_error_dim_16_lin_p_same_orient_LISG.txt};
        \addplot[
            only marks,
            mark=x,
            mark options={scale=3, line width = 2pt},
            color = skyblue,
            forget plot
        ] table [
            x index=1,
            y index=0,
            col sep=space,
            skip first n=9
        ] {author/L2_error_dim_16_lin_p_same_orient_LISG.txt};
        \addplot table [x index=1,y index=0, col sep=space] {author/L2_error_dim_16_lin_p_same_orient_ASG.txt};
        \addplot table [x index=1,y index=0, col sep=space] {author/L2_error_dim_16_lin_p_same_orient_ISG.txt};

        \legend{DASG, LISG, ASG, ISG};
        \end{axis}
        \node[rotate=90] at (-1.3,2.8) {\large Approx. \(L^2\)-Error};
    \end{tikzpicture}
    \end{subfigure}
    \hspace{10mm}
    \begin{subfigure}{.45\textwidth}
    \begin{tikzpicture}[scale=0.75]
        \begin{axis}[
        cycle list name = myCycleList,
        grid = both,
        xmode = log, ymode = log,
        xmin = 0.8, xmax = 1e5,
        ymin = 1e+-10, ymax= 1,
        xlabel={\Large$N$},
        legend pos=south west,
        yticklabels = {}
        ]
        \addplot table [x index=1,y index=0, col sep=space] {author/L2_error_dim_16_lin_p_opposite_orient_plus_0.txt};
        \addplot[
            only marks,
            mark=x,
            mark options={scale=3, line width = 2pt},
            color = darkgray,
            forget plot
        ] table [
            x index=1,
            y index=0,
            col sep=space,
            skip first n=20
        ] {author/L2_error_dim_16_lin_p_opposite_orient_plus_0.txt};
        \addplot table [x index=1,y index=0, col sep=space] {author/L2_error_dim_16_lin_p_opposite_orient_LISG.txt};
        \addplot[
            only marks,
            mark=x,
            mark options={scale=3, line width = 2pt},
            color = skyblue,
            forget plot
        ] table [
            x index=1,
            y index=0,
            col sep=space,
            skip first n=8
        ] {author/L2_error_dim_16_lin_p_opposite_orient_LISG.txt};
        \addplot table [x index=1,y index=0, col sep=space] {author/L2_error_dim_16_lin_p_opposite_orient_ASG.txt};
        \addplot table [x index=1,y index=0, col sep=space] {author/L2_error_dim_16_lin_p_same_orient_ISG.txt};

        \legend{DASG, LISG, ASG, ISG};
        \end{axis}
        \node[rotate=-90] at (7.7,2.8) {\Large$d=16$};
    \end{tikzpicture}
    \end{subfigure}
    \caption{Mean \(L^2\)-approximation error of separable Mat\'ern kernel interpolants of realisations of \(f_{\Phi,\boldsymbol{\nu},\mathbf{p}}\), constructed on sparse grid designs of increasing level, \(L\). Both doubly anisotropic sparse grid (DASG) and lengthscale-informed sparse grid (LISG) methods employ kernels with anisotropic lengthscales, specified by \(\boldsymbol{\lambda}=2^{\mathbf{p}_{\textrm{lin}}}\), whereas the anisotropic sparse grid (ASG) and isotropic (ISG) methods use \(\boldsymbol{\lambda}=\mathbf{1}\). All kernels correctly specify the smoothness anisotropy via \(\boldsymbol{\nu}\). {An \(\times\) indicates that the Gram matrix of the subsequent interpolant was non positive-definite in double precision.}
    }
    \label{fig: DASG vs LISG}
\end{figure}
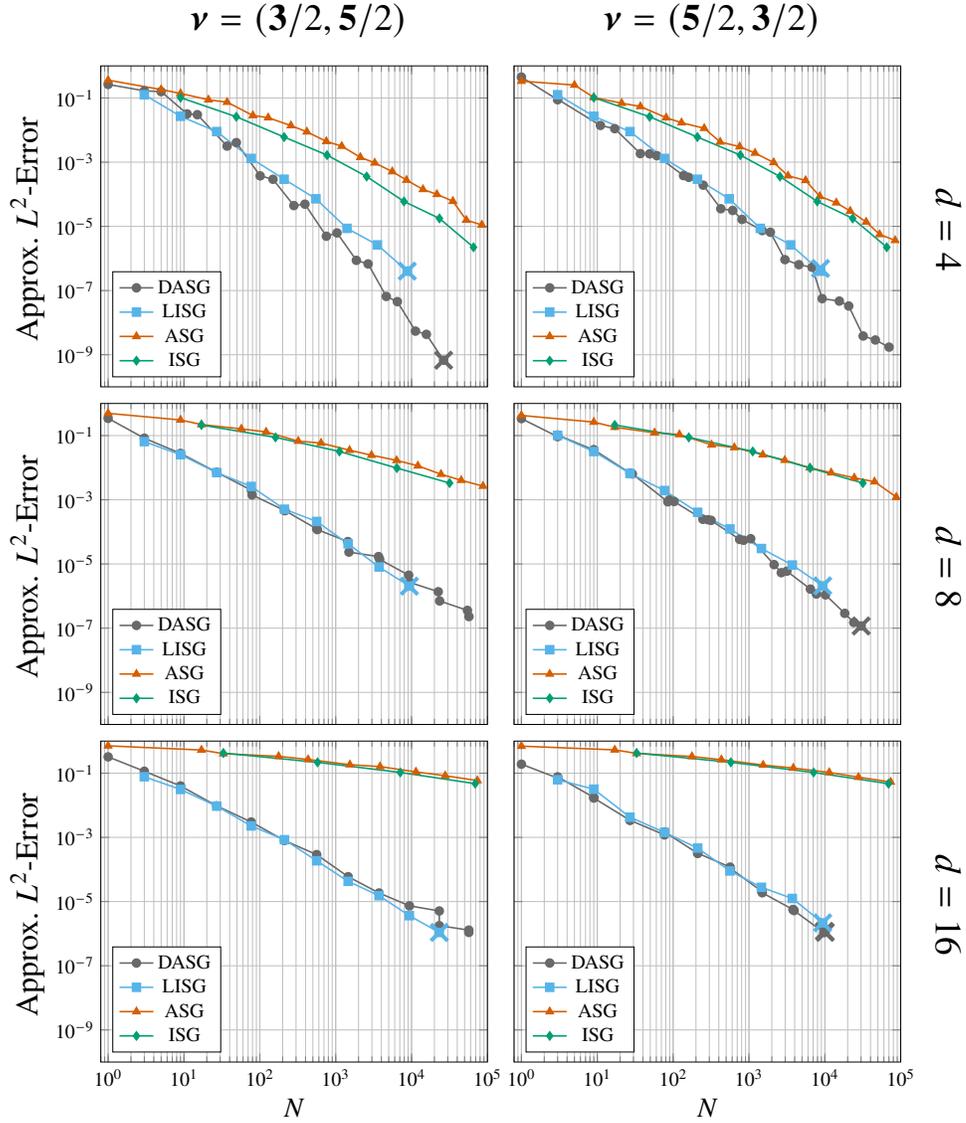

\begin{acknowledgement}
EJA was supported by Biomathematics and Statistics Scotland (BioSS) and the EPSRC Centre
for Doctoral Training in Mathematical Modelling, Analysis and Computation (MAC-MIGS) funded by
the EPSRC (EP/S023291/1), Heriot-Watt University, and The University of Edinburgh. {ALT was partially
supported by EPSRC grants EP/X01259X/1 and EP/Y028783/1.}
\end{acknowledgement}

\printbibliography

\end{document}